\documentclass[10pt]{article}
\usepackage[margin=1in]{geometry}
\makeatletter
\usepackage{fullpage}

\usepackage{amssymb, amsmath, amsthm, latexsym}

\usepackage{url}
\usepackage{algorithm}
\usepackage{algorithmic}
\usepackage{tabularx}
\usepackage{booktabs}
\usepackage{paralist}
\usepackage{mathtools}
\usepackage{xcolor}
\usepackage{hyperref}
\usepackage[capitalise]{cleveref}

\usepackage{bbm}
\usepackage{bm}
\usepackage{wrapfig}
\usepackage{makecell}
\usepackage{multirow}

\usepackage{nicefrac}

\usepackage{float}
\usepackage[flushleft]{threeparttable}

\usepackage{caption}
\usepackage{multirow}
\usepackage{colortbl}
\definecolor{bgcolor}{rgb}{0.8,1,1}
\definecolor{bgcolor2}{rgb}{0.8,1,0.8}
\definecolor{niceblue}{rgb}{0.0,0.19,0.56}

\hypersetup{colorlinks,linkcolor={blue},citecolor={niceblue},urlcolor={blue},
  pdftitle={Acceleration Survives Heavy Tails and Generalized Smoothness},
  pdfauthor={Taha El Bakkali El Kadi, Eduard Gorbunov}}

\usepackage{natbib}
\usepackage{sidecap}

\newtheorem{theorem}{Theorem}
\newtheorem{proposition}{Proposition}
\newtheorem{corollary}{Corollary}

\newtheorem{assumption}{Assumption}

\newtheorem{remark}{Remark}
\newtheorem{lemma}{Lemma}
\crefname{assumption}{Assumption}{Assumptions}
\Crefname{assumption}{Assumption}{Assumptions}
\newcommand*{\R}{{\mathbb R}}
\newcommand*{\E}{{\mathbb E}}

\newcommand{\1}{\mathbbm{1}}
\newcommand{\ip}[2]{\left\langle #1,#2\right\rangle}

\renewcommand{\|}{\parallel}

\newcommand\ev[1]{ \left\langle #1 \right\rangle}

\newcommand{\norm}[1]{\left\lVert#1\right\rVert}

\definecolor{tiffanyblue}{rgb}{0.04, 0.73, 0.71}
\usepackage[most]{tcolorbox}
\newtcolorbox{remarkbox}{
    colback=tiffanyblue!20,
    colframe=black,
    boxrule=0.5mm,
    sharp corners,
    left=2mm,
    right=2mm,
    top=2mm,
    bottom=2mm
}

\newtcolorbox{remarkbox2}{
    colback=tiffanyblue!20,
    colframe=black,
    boxrule=0.5mm,
    sharp corners,
    left=2mm,
    right=2mm,
    top=0mm,
    bottom=0mm
}

\usepackage{pifont}
\definecolor{PineGreen}{RGB}{0,110,51}
\definecolor{BrickRed}{RGB}{143,20,2}

\usepackage{tikz-cd}

\usepackage{graphicx}

\newcolumntype{Y}{>{\centering\arraybackslash}X}

\usepackage{xspace}

\definecolor{mycolor}{RGB}{0,150,70}
\newcommand{\algname}[1]{{\color{PineGreen}\sf  #1}\xspace}

\usepackage[colorinlistoftodos,bordercolor=orange,backgroundcolor=orange!20,linecolor=orange,textsize=scriptsize]{todonotes}

\def\clip{\texttt{clip}}

\usepackage{amsthm}
\usepackage{hyperref}
\newcommand{\mytag}[1]{%
  \refstepcounter{equation}%
  \edef\@currentlabel{\theequation}%
  {({\@currentlabel})}%
  \@bsphack
  \begingroup
    \@onelevel@sanitize\@currentlabelname
    \edef\@currentlabelname{%
      \expandafter\strip@period\@currentlabelname\relax.\relax\@@@%
    }%
    \protected@write\@auxout{}{%
      \string\newlabel{#1}{%
        {\@currentlabel}%
        {\thepage}%
        {\@currentlabelname}%
        {\@currentHref}{}%
      }%
    }%
  \endgroup
  \@esphack
}
\makeatother

\title{\bf Acceleration Survives Heavy Tails and Generalized Smoothness}
\author{
\begin{tabular}[t]{c}
     {\bf Taha El Bakkali El Kadi}\\
     {\small MBZUAI\textsuperscript{$\dagger$}, UM6P} \\
     {\small \texttt{taha.elbakkali@um6p.ma}}
\end{tabular}
\quad
\begin{tabular}[t]{c}
     {\bf Eduard Gorbunov}\\
     {\small MBZUAI}\\
     {\small \texttt{eduard.gorbunov@mbzuai.ac.ae}}
\end{tabular}
}

\begin{document}

\maketitle
\begingroup
\renewcommand{\thefootnote}{\fnsymbol{footnote}}
\footnotetext[2]{This work was done when Taha El Bakkali El Kadi was a visiting student in the group of Prof.~Eduard Gorbunov at MBZUAI.}

\endgroup
\begin{abstract}
 We establish high-probability accelerated convergence guarantees for convex $(L_0,L_1)$-smooth optimization with unbiased stochastic gradients under heavy-tailed noise, assuming a central $\alpha$-th noise moment bounded by $\sigma^\alpha$, $\alpha\in(1,2]$. Such guarantees are known under classical $L$-smoothness, but extending them through localization over an $O(R)$-radius ball may introduce an $e^{O(L_1R)}$ factor. We avoid this dependence in two ways. First, a two-phase scheme uses \algname{Clip-SGD} to reach a low-curvature regime, after which clipped stochastic similar triangles (\algname{Clip-SSTM}) attains, with probability at least $1-\delta$,
$
f(y_K)-f^\star\le\widetilde{\mathcal O}\big(L_0R^2/K^2+\sigma RK^{-(\alpha-1)/\alpha}\big),
$
where $R$ bounds the distance from the Phase-II starting point to a minimizer. Second, a modified deterministic weight schedule allows \algname{Clip-SSTM} to start directly from the initial point, controlling the high-curvature transient before recovering acceleration. With the prescribed initialization, both approaches recover the classical accelerated rate, with $L_1$ confined to an accuracy-independent transient, up to logarithmic factors. The two-phase analysis uses two coupled potentials to preserve the low-curvature regime despite momentum; the one-phase analysis controls the initial high-curvature regime using a modified version of the classical SSTM weight schedule.
\end{abstract}

\tableofcontents

\section{Introduction}

Accelerated stochastic optimization is well understood under classical smoothness, where the gradient is globally Lipschitz. Two additional difficulties arise in settings relevant to modern learning. First, the local smoothness of the objective may grow with the gradient norm, a dependence observed by \citet{zhang2019gradient} along training trajectories. Second, stochastic gradient noise can be heavy-tailed, as observed, for example, in BERT pre-training \citep{zhang2020adaptive}. Together, these effects make acceleration challenging: high curvature may require smaller steps, while large stochastic gradient errors can be amplified by momentum. Gradient clipping can limit the effect of such errors, but introduces bias that must also be controlled. Under classical $L$-smoothness, high-probability accelerated guarantees under heavy-tailed noise have been established by \citet{sadiev2023high,nguyen2023improved}.

The $(L_0,L_1)$-smoothness condition provides a natural model for such nonuniform curvature. For twice-differentiable objectives, this dependence is captured by $\norm{\nabla^2 f(x)}\le L_0+L_1\norm{\nabla f(x)}$ \citep{zhang2019gradient}. First-order formulations for continuously differentiable functions were studied by \citet{zhang2020improved} and \citet{chen2023generalized}; we adopt the formulation in Assumption~\ref{asm: smoothness}. In this model, the curvature can be large when the gradient norm is large, while becoming comparable to $L_0$ when $\norm{\nabla f(x)}$ is of order $L_0/L_1$.

A natural approach is to take an accelerated method with guarantees under classical smoothness, show that its iterates remain in a ball of radius $O(R)$ \citep{sadiev2023high}, and use $(L_0,L_1)$-smoothness to obtain an effective smoothness constant over this ball. One could then apply the standard accelerated analysis with this local constant. However, the resulting smoothness parameter can be as large as $L_0\exp(O(L_1R))$, leading to an exponential dependence on $L_1R$.

Deterministic acceleration under generalized smoothness has been
studied in several recent works.
\citet{gorbunov2025methods} proposed an accelerated Similar Triangles
method for convex $(L_0,L_1)$-smooth optimization, although its
convergence guarantee retains an exponential dependence on $L_1R$,
where $R$ bounds the initial distance to a minimizer.
\citet{vankov2025optimizing} subsequently established accelerated
guarantees without such an exponential dependence, using a method
that requires solving a one-dimensional line-search subproblem at
each iteration, making a direct extension to stochastic gradients
nontrivial. More recently, \citet{pmlr-v306-tyurin26b} obtained
deterministic accelerated guarantees under an $\ell$-smoothness
assumption that generalizes $(L_0,L_1)$-smoothness.
They consider both a gradient-descent warm start followed by
acceleration and a method that starts directly from the initial
point, adjusting its stepsizes to control the initial high-curvature
regime. Under $(L_0,L_1)$-smoothness, these methods recover an
$L_1$-independent leading high-accuracy rate after an initial
transient. Our two-phase and one-phase approaches follow similar
principles, but use \algname{Clip-SSTM} with different weight
choices. The additional challenge is to control curvature and
accumulated clipping errors simultaneously along the accelerated
trajectory, with high probability under heavy-tailed noise.

Stochastic accelerated-gradient methods under $(L_0,L_1)$-smoothness
have also been analyzed recently.
\citet{yu2025convergence} study randomized stochastic accelerated
gradient methods under $(L_0,L_1)$-smoothness and relaxed affine
noise assumptions, obtaining high-probability convergence guarantees
for convex optimization; their analysis, however, does not yield
an accelerated convergence rate in this setting.
\citet{dvinskikh2026localize} establish high-probability accelerated
guarantees for convex $(L_0,L_1)$-smooth optimization under unbiased
stochastic gradients with additive norm-sub-Gaussian noise,
while avoiding an exponential dependence on $L_1R$.
Their guarantees assume light-tailed noise, whereas we assume
only a bounded central $\alpha$-th moment with $\alpha\in(1,2]$, which is a weaker condition.

This gap raises a natural question:
\begin{center}
\begin{tcolorbox}[
    width=0.95\linewidth,
    colback=PineGreen!5,
    colframe=PineGreen!50,
    boxrule=0.5pt,
    arc=1mm,
    left=3mm, right=3mm,
    top=2mm, bottom=2mm
]
\centering\itshape
Can classical clipped accelerated schemes achieve
high-probability acceleration under $(L_0,L_1)$-smoothness
and heavy-tailed noise without exponential dependence on $L_1R$?
\end{tcolorbox}
\end{center}

We answer affirmatively.

\paragraph{Our contributions.}
Under Assumptions~\ref{asm: convexity} (convexity), \ref{asm: smoothness} ($(L_0,L_1)$-smoothness),
and \ref{asm:oracle} (bounded central $\alpha$-th moment, $\alpha\in(1,2]$), we establish high-probability accelerated
guarantees for two approaches.
Let $\norm{x_0-x^\star}\le R_0$.
Our total complexity bounds avoid exponential dependence on $L_1R_0$.

\begin{itemize}
\item \textbf{Two-phase acceleration.}
We use a nonaccelerated method as a warm start to obtain
a point $\bar x$ satisfying $f(\bar x)-f^\star\le cL_0/L_1^2$,
an accuracy-independent threshold implying
$\norm{\nabla f(\bar x)}=O(L_0/L_1)$ with high probability.
Starting from $\bar x$, \algname{Clip-SSTM} satisfies,
with probability at least $1-\delta$,
$$
f(y_K)-f^\star
\le\widetilde{\mathcal O}\!\left(
L_0R^2/K^2+\sigma R K^{-(\alpha-1)/\alpha}\right)
$$
after $K$ additional stochastic gradient evaluations,
where $\norm{\bar x-x^\star}\le R$ and
$\widetilde{\mathcal O}$ hides logarithmic factors in $K$ and $1/\delta$.
The analysis couples two potentials, anchored at $\bar x$ and
$x^\star$, to keep the accelerated trajectory in the
low-curvature region despite momentum.

\item \textbf{One-phase acceleration.}
\algname{Clip-SSTM} can also start directly from $x_0$,
using a modified deterministic weight schedule to control
the high-curvature transient before recovering acceleration.
With the prescribed initialization, it recovers the classical
$L$-smooth accelerated rate after a transient that is
accuracy-independent up to logarithmic factors.
Here the distance bound is $R_0$.
\end{itemize}

\paragraph{Complexity guarantees.}
For the two-phase method, the warm start also guarantees $\norm{\bar x-x^\star}\le 2R_0$ on its success event, so the accelerated phase can use $R=2R_0$. Including the warm-start cost, \cref{cor:end-to-end} ensures $f(y)-f^\star\le\varepsilon$ with probability at least $1-\delta$
using $\widetilde{\mathcal O}\left(1+(L_1R_0)^2+\allowbreak \sqrt{\frac{L_0R_0^2}{\varepsilon}}+\allowbreak\bigl(\frac{\sigma R_0}{\varepsilon}\bigr)^{\frac{\alpha}{\alpha-1}}
\right)$ stochastic gradient evaluations. For the one-phase method, \cref{thm:one-phase} with the
initialization in \cref{rem_init} ensures $f(y_K)-f^\star\le\varepsilon$ with the same probability using $\widetilde{\mathcal O}\left(K_0+\allowbreak\sqrt{\frac{L_0R_0^2}{\varepsilon}}+\allowbreak
\bigl(\frac{\sigma R_0}{\varepsilon}\bigr)^{\frac{\alpha}{\alpha-1}}\right)$ stochastic gradient evaluations. The transient $K_0$ depends logarithmically on the initial-gap bound.
Using the worst-case gap bound gives $K_0=\widetilde{\mathcal O}((1+L_1R_0)^2)$;
a tighter bound can give a smaller transient. Up to logarithmic factors and an additional initial transient, both methods recover the complexity of accelerated methods
for classical $L$-smooth convex optimization under heavy-tailed noise, with $L$ replaced by $L_0$.
The dependence on $L_1$ is confined to the transient.
\begin{table}[H]
\centering
\caption{
Comparison with related results for convex optimization.
HP denotes a high-probability guarantee.
Here $\alpha$-moment means bounded central $\alpha$-th moment noise,
$\alpha\in(1,2]$, and ``No exp. dep.'' means no exponential dependence
on $L_1R$ (or $\sqrt{H_1}R$ under $(H_0,H_1)$-smoothness).
}
\label{tab:comparison}
\scriptsize
\setlength{\tabcolsep}{4pt}
\renewcommand{\arraystretch}{0.95}

\begin{tabular}{lccccc}
\toprule
\textbf{Reference}
& \textbf{Smoothness}
& \textbf{Noise}
& \textbf{HP}
& \textbf{Accel. rate}
& \textbf{No exp. dep.}
\\
\midrule

\rowcolor{gray!10}
\multicolumn{6}{c}{
\textit{Classical $L$-smooth optimization under heavy-tailed noise}
}
\\
\midrule

\citet{sadiev2023high}
& $L$-smooth
& $\alpha$-moment
& \ding{51}
& \ding{51}
& N/A
\\

\citet{nguyen2023improved}
& $L$-smooth
& $\alpha$-moment
& \ding{51}
& \ding{51}
& N/A
\\

\midrule
\rowcolor{gray!10}
\multicolumn{6}{c}{
\textit{$(L_0,L_1)$-smooth optimization: nonaccelerated guarantees}
}
\\
\midrule

\citet{gorbunov2025methods} (nonaccel.)
& $(L_0,L_1)$
& Det.
& N/A
& \ding{55}
& \ding{51}
\\

\citet{vankov2025optimizing} (nonaccel.)
& $(L_0,L_1)$
& Det.
& N/A
& \ding{55}
& \ding{51}
\\

\citet{lobanov2024linear}
& $(L_0,L_1)$
& Det.
& N/A
& \ding{55}
& \ding{51}
\\

\citet{gaash2025convergence}
& $(L_0,L_1)$
& sub-Gaussian
& \ding{51}
& \ding{55}
& \ding{51}
\\

\citet{kadi2026highprobabilityconvergenceclippedsgd}
& $(L_0,L_1)$
& $\alpha$-moment
& \ding{51}
& \ding{55}
& \ding{51}
\\

\midrule
\rowcolor{gray!10}
\multicolumn{6}{c}{
\textit{Generalized-smooth optimization: acceleration and momentum methods}
}
\\
\midrule

\citet{gorbunov2025methods} (accelerated)
& $(L_0,L_1)$
& Det.
& N/A
& \ding{51}
& \ding{55}
\\

\citet{vankov2025optimizing} (accelerated)$^\dagger$
& $(L_0,L_1)$
& Det.
& N/A
& \ding{51}
& \ding{51}
\\

\citet{pmlr-v306-tyurin26b}
& $\ell$-smooth
& Det.
& N/A
& \ding{51}
& \ding{51}
\\

\citet{yu2025convergence}$^\ddagger$
& $(L_0,L_1)$
& relaxed affine noise
& \ding{51}
& \ding{55}
& \ding{51}
\\

\citet{dvinskikh2026localize}
& $(L_0,L_1)$
& norm-sub-Gaussian
& \ding{51}
& \ding{51}
& \ding{51}
\\

\citet{lobanov2026acceleratedstochasticmethodh0h1smoothness}$^*$
& $(H_0,H_1)$
& generalized $\alpha$-moment
& \ding{51}
& \ding{51}
& \ding{51}
\\

\rowcolor{blue!5}
\textbf{This work}
& $(L_0,L_1)$
& $\alpha$-moment
& \ding{51}
& \ding{51}
& \ding{51}
\\

\bottomrule
\end{tabular}

\vspace{2pt}
\begin{minipage}{0.97\textwidth}
\footnotesize
Det. = deterministic.
$^\dagger$\citet{vankov2025optimizing} use a one-dimensional line search
with function values for their accelerated method.
$^\ddagger$\citet{yu2025convergence} analyze stochastic accelerated-gradient schemes under relaxed affine almost-sure noise, with a sub-Gaussian extension. Their convex high-probability guarantee does not establish an accelerated
rate. 
$^*$Concurrent work.
\citet{lobanov2026acceleratedstochasticmethodh0h1smoothness}
consider $(H_0,H_1)$-smoothness and a generalized central moment
bound with constant, gradient-dependent, and objective-gap-dependent
terms. Their method uses projections and phase restarts;
our methods require neither projections nor repeated restarts.
See the \hyperref[par:concurrent-work]{discussion of concurrent work}
for details.
\end{minipage}

\end{table}

\section{Preliminaries}
\label{sec:prel}

\paragraph{Notation.}
We work in $\mathbb R^d$, equipped with the standard Euclidean
inner product $\langle\cdot,\cdot\rangle$ and norm $\norm{\cdot}$.
For $\lambda>0$, let
$\mathbb B(0,\lambda):=\{x\in\mathbb R^d:\norm{x}\le\lambda\}$.
We define the clipping operator $\clip(x,\lambda)$ as the Euclidean projection
onto this ball:
$$
\clip(x,\lambda)
:=\Pi_{\mathbb B(0,\lambda)}(x)
=
\begin{cases}
x, & \norm{x}\le\lambda,\\[2pt]
\dfrac{\lambda}{\norm{x}}x, & \norm{x}>\lambda.
\end{cases}
$$
Thus, clipping preserves the direction of a nonzero vector
while limiting its norm to $\lambda$.
We assume that $f$ admits a minimizer, fix
$x^\star\in\arg\min_{x\in\mathbb R^d}f(x)$, and denote
$f^\star:=f(x^\star)$.
We use $L_0,L_1>0$ for the constants in the
$(L_0,L_1)$-smoothness condition introduced below.
Finally, let $\nu\in(0,1)$ be the unique positive solution
of $\nu e^\nu=1$.

\paragraph{Problem.}
We minimize $f$ over $\mathbb R^d$ using stochastic gradient
samples $\nabla f(x,\xi)$, under the assumptions stated below.
We establish high-probability objective-gap bounds for
\algname{Clip-SSTM} with two initialization and weight choices:
standard accelerated weights after a warm start in a
low-curvature regime, and modified weights that allow a direct
start from $x_0$ without a small-gap requirement.

\paragraph{Common algorithm.}
Both procedures use Algorithm~\ref{alg:Clip-sstm},
where $v_0$ is the starting point, $A_0\ge0$ the initial
accumulated weight, $\alpha_{k+1}>0$ the added weight,
and $\lambda_{k+1}>0$ the clipping level.
Their initialization and weight choices are specified below.

\begin{algorithm}[H]
\caption{\algname{Clip-SSTM}}
\label{alg:Clip-sstm}
\begin{algorithmic}[1]
\STATE \textbf{Input:} Starting point $v_0$, iteration budget $K\ge1$, initial weight $A_0\ge0$,
positive weights $\{\alpha_k\}_{k=1}^K$, and positive clipping levels
$\{\lambda_k\}_{k=1}^K$.
\STATE $y_0=z_0=v_0$.
\FOR{$k=0,\ldots,K-1$}
    \STATE $\displaystyle
    A_{k+1}=A_k+\alpha_{k+1}$.
    \STATE $\displaystyle
    x_{k+1}
    =
    \frac{A_ky_k+\alpha_{k+1}z_k}{A_{k+1}}$.
    \STATE Draw a fresh sample $\xi_{k+1}$, independent of the past.
    \STATE $\displaystyle
    g_{k+1}
    =
    \clip\!\left(
    \nabla f(x_{k+1},\xi_{k+1}),
    \lambda_{k+1}
    \right)$.
    \STATE $\displaystyle
    z_{k+1}=z_k-\alpha_{k+1}g_{k+1}$.
    \STATE $\displaystyle
    y_{k+1}
    =
    \frac{A_ky_k+\alpha_{k+1}z_{k+1}}{A_{k+1}}$.
\ENDFOR
\STATE \textbf{Output:} $y_K$.
\end{algorithmic}
\end{algorithm}

\paragraph{Two-phase procedure.}
Phase~I runs a warm-start method and returns a point $\bar x$. On its success event,
$$
\begin{cases}
\norm{\bar x-x^\star}\le R,\\[0.5em]
f(\bar x)-f^\star\le c\dfrac{L_0}{L_1^2},
\end{cases}
$$
where $c>0$ is a universal constant specified in the corresponding convergence theorem.

Phase~II runs \algname{Clip-SSTM} using fresh stochastic-gradient samples. For a parameter $a>0$, we initialize $y_0=z_0=\bar x$, $A_0=0$, and, for $k\ge0$, choose $
\alpha_{k+1}=\frac{k+2}{aL_0}.$ For this choice, $A_k=\frac{k(k+3)}{2aL_0},$ and $A_{k+1}\ge\frac{aL_0}{2}\alpha_{k+1}^2.$ Thus $A_k$ grows quadratically from the beginning of Phase~II. The small-gap condition at $\bar x$ places the accelerated method in a local low-curvature regime, allowing us to use from the first iteration the standard accelerated weight schedule of \cite{sadiev2023high}.

\paragraph{One-phase procedure.}
The one-phase procedure runs \algname{Clip-SSTM} directly from $x_0$. We assume
$\norm{x_0-x^\star}\le R_0,$
but impose no small-gap condition. Let $\Delta_0:=f(x_0)-f^\star$, and let $\bar\Delta_0\ge\Delta_0$ be any deterministic upper bound on the initial objective gap.

Let $a>0$ and $\ell\ge1$ be parameters, whose specific choices are given in the corresponding convergence theorem; in particular, the theorem sets $\ell=8\log(4K/\delta)$. We initialize $y_0=z_0=x_0$ and choose $A_0>0$ satisfying $A_0\le\frac{1}{aL_0}$ and $
A_0\bar\Delta_0\le\frac{R_0^2}{64}.$ In particular, $A_0\bigl(f(x_0)-f^\star\bigr)\le A_0\bar\Delta_0\le\frac{R_0^2}{64},$ so the initial objective gap contributes only a controlled amount to the potential used in the analysis. For $k\ge0$, we use the weight rule
$$
\alpha_{k+1}=\min\left\{\frac{A_k}{64\ell(1+L_1R_0)},\sqrt{\frac{A_k}{aL_0}}\right\}.
$$
Consequently, $\frac{\alpha_{k+1}}{A_k}\le\frac{1}{64\ell(1+L_1R_0)}$ and $
A_{k+1}\ge A_k\ge aL_0\alpha_{k+1}^2.$ The two terms in the minimum correspond to two regimes. If
$$
A_k<\frac{[64\ell(1+L_1R_0)]^2}{aL_0},
$$
then $\frac{A_k}{64\ell(1+L_1R_0)}<\sqrt{\frac{A_k}{aL_0}},$ so the method uses
$
\alpha_{k+1}=\frac{A_k}{64\ell(1+L_1R_0)}.
$
Thus, in this regime, the weight $\alpha_{k+1}$ is deliberately chosen smaller than the standard accelerated square-root weight $\sqrt{A_k/(aL_0)}$.

Once
$$
A_k\ge\frac{[64\ell(1+L_1R_0)]^2}{aL_0},
$$
the square-root term becomes the smaller one, and hence $\alpha_{k+1}=\sqrt{\frac{A_k}{aL_0}}.$ Since $A_{k+1}\ge A_k$, the square-root choice remains active thereafter. The method then follows the standard accelerated square-root schedule $
A_{k+1}=A_k+\sqrt{\frac{A_k}{aL_0}},$ for which $A_k$ grows quadratically with the number of subsequent iterations.

Thus, the modified weight rule uses smaller-than-standard accelerated weights during the high-curvature transient and automatically switches to the standard accelerated schedule once this additional restriction is no longer active. The remaining parameter choices and clipping levels for the two procedures are specified in their respective convergence theorems.
\subsection{Assumptions}

We first assume convexity.

\begin{assumption}[Convexity]
\label{asm: convexity}
The function $f$ is differentiable and convex, i.e.,
$$
f(y)\ge f(x)+\ip{\nabla f(x)}{y-x}  \quad \forall x,y\in \R^d.
$$
\end{assumption}

We next assume $(L_0,L_1)$-smoothness.

\begin{assumption}[$(L_0,L_1)$-smoothness]
\label{asm: smoothness}
There exist constants $L_0,L_1\ge0$ such that, for all
$x,y\in\mathbb R^d$,
$$
\norm{\nabla f(x)-\nabla f(y)}\le \left(L_0+L_1 \sup_{u\in[x,y]}\norm{\nabla f(u)}\right)\norm{x-y},
$$
\end{assumption}
This assumption is strictly more general than the classical $L$-smoothness assumption, which requires $\nabla f$ to be globally Lipschitz continuous: it reduces to the standard $L$-smoothness with $L = L_0$ if $L_1 = 0$. Moreover, one can construct functions that satisfy Assumption~\ref{asm: smoothness} but not $L$-smoothness, e.g., exponent of norm $f(x) = \exp(\norm{x}) + \exp(-\norm{x})$, power of norm $f(x) = \norm{x}^{n}$ for $n>2$, and exponent of the linear function $f(x) = \exp(\ev{a,x})$ \citep{chen2023generalized, gorbunov2025methods}.

Finally, we assume unbiasedness and boundedness of the
central $\alpha$-th moment.

\begin{assumption}[Stochastic oracle]
\label{asm:oracle}
There exist $\alpha\in(1,2]$ and $\sigma\ge0$ such that,
for every $x\in\mathbb R^d$,
$$\begin{cases}
\mathbb E_\xi[\nabla f(x,\xi)] = \nabla f(x),\\[0.5em]
\mathbb E_\xi\!\left[
\norm{\nabla f(x,\xi)-\nabla f(x)}^\alpha
\right]
\le \sigma^\alpha.
\end{cases}$$
\end{assumption}
This assumption has become relatively standard and has been considered in
\citep{zhang2020adaptive,cutkosky2021high,sadiev2023high,
nguyen2023improved,chezhegov2024gradient}.
For $\alpha<2$, examples include stable distributions with stability parameter
$\beta\in(\alpha,2)$.
When $\alpha=2$, \cref{asm:oracle} reduces to the classical bounded-variance
assumption \citep{nemirovski2009robust,ghadimi2013stochastic,takavc2013mini}.
\section{Main Results}
\label{sec:main_results}

\subsection{Two-phase result}
\label{sec:two-phase-result}

We first analyze the two-phase procedure. The following theorem
gives the high-probability convergence guarantee for the accelerated
stage initialized from a point satisfying the warm-start conditions.

\begin{theorem}
\label{thm:main}
Suppose that Assumptions~\ref{asm: convexity}, \ref{asm: smoothness}, and
\ref{asm:oracle} hold, where $L_0,L_1>0$. Assume that $\norm{\bar x-x^\star}\le R$ and
$f(\bar x)-f^\star\le cL_0/L_1^2$, where $R>0$ and $0<c\le1/32$. Let $K\ge1$, $\delta\in(0,1)$, and set $\ell:=8\log(4K/\delta)$ and
$q:=1024^{1/\alpha}\ell^{(\alpha-1)/\alpha}K^{1/\alpha}$. Let
$a:=\max\left\{192\ell^2,\frac{2q^2}{c}\left(\frac{\sigma L_1}{L_0}\right)^2,\frac{2q\sigma K}{L_0R}\right\}.$
Run \algname{Clip-SSTM} with starting point $v_0=\bar x$, initialization $A_0=0$ and $y_0=z_0=\bar x$, and, for every $k=0,\ldots,K-1$, let $\alpha_{k+1}:=\frac{k+2}{aL_0}.$
For $i=1,\ldots,K$, let $\lambda_i:=\frac1{\alpha_i\ell}\min\left\{\frac{\sqrt{cA_iL_0}}{L_1},R\right\}.$
Then, with probability at least $1-\delta$, the output $y_K$ of \algname{Clip-SSTM} satisfies
$$
f(y_K)-f^\star=\widetilde{\mathcal O}\left(\frac{L_0R^2}{K^2}+\frac{\sigma R}{K^{(\alpha-1)/\alpha}}\right),
$$
where $\widetilde{\mathcal O}$ suppresses logarithmic factors in $K$ and $1/\delta$.
\end{theorem}
\paragraph{A direct warm-start bound.}
The warm-start condition required in Theorem~\ref{thm:main} can be
obtained using \algname{Clip-SGD}. Let $x_0$ satisfy
$\|x_0-x^\star\|\le R_0$, let
$\Delta_0:=f(x_0)-f^\star$, and allocate failure probability $\delta/2$
to Phase~I. Define
$
M:=\max\left\{L_1R_0,\log\frac{24}{\delta}\right\}.
$
By \citet[Theorem~2]{kadi2026highprobabilityconvergenceclippedsgd}, under Assumptions~\ref{asm: convexity}, \ref{asm: smoothness}, and \ref{asm:oracle}, \algname{Clip-SGD} produces, after $K_1$ stochastic
gradient evaluations, a computable output $\bar x$ satisfying, with
probability at least $1-\delta/2$,
$$
f(\bar x)-f^\star=\mathcal O\left(\max\left\{\frac{L_0R_0^2}{K_1},R_0\sigma
\left(\frac{M}{K_1}\right)^{\frac{\alpha-1}{\alpha}}\right\}\right).
$$
An explicit sufficient iteration requirement, as stated in
\citet{kadi2026highprobabilityconvergenceclippedsgd}, is
$$
K_1=\Omega\left(\log^2\frac{24}{\delta}+L_1R_0\log\left[\frac{24e}{\delta}
\left(1+\frac{L_1\Delta_0}{L_0R_0}\right)\right]\right).
$$
To reach the switching accuracy
$f(\bar x)-f^\star\le cL_0/L_1^2$, the two terms in the convergence
bound require, up to universal constants, $K_1=\Omega\left((L_1R_0)^2\right)$
and $K_1=\Omega\left(M\left(\frac{\sigma R_0L_1^2}{L_0}\right)^{\frac{\alpha}{\alpha-1}}
\right),$ respectively. Moreover, the additional iteration requirement above is already covered,
up to logarithmic factors, by $1+(L_1R_0)^2$. Consequently, Phase~I reaches a point satisfying
$f(\bar x)-f^\star\le cL_0/L_1^2$ after $
K_1=\widetilde{\mathcal O}\left(1+(L_1R_0)^2+M\left(\frac{\sigma R_0L_1^2}{L_0}
\right)^{\frac{\alpha}{\alpha-1}}\right)$ stochastic gradient evaluations.

Moreover, on the same event, the localization argument in the proof of
\citet[Theorem~2]{kadi2026highprobabilityconvergenceclippedsgd} gives
$\|x_k-x^\star\|\le 2R_0$ for all Phase~I iterates. Since $\bar x$  is a convex combination of a subset of these iterates, it follows that
$\|\bar x-x^\star\|\le 2R_0$.
Hence Theorem~\ref{thm:main} can be applied in Phase~II with
$R=2R_0$ and failure probability $\delta/2$. Therefore, to obtain
$f(y_K)-f^\star\le\varepsilon$, Phase~II requires
$
K=\widetilde{\mathcal O}\left(\sqrt{\frac{L_0R_0^2}{\varepsilon}}+\left(
\frac{\sigma R_0}{\varepsilon}\right)^{\frac{\alpha}{\alpha-1}}\right).$

By a union bound, the complete two-phase procedure succeeds with
probability at least $1-\delta$, and its total stochastic-gradient
complexity is
$$\widetilde{\mathcal O}\left(1+(L_1R_0)^2+M\left(\frac{\sigma R_0L_1^2}{L_0}\right)^{\frac{\alpha}{\alpha-1}} +\sqrt{\frac{L_0R_0^2}{\varepsilon}}+\left(\frac{\sigma R_0}{\varepsilon}\right)^{\frac{\alpha}{\alpha-1}}\right).$$
In particular, all explicit $L_1$-dependence is confined to the accuracy-independent Phase-I cost. 

We can further sharpen the transient term and remove its explicit dependence on $\sigma$. The key observation is that $(L_0,L_1)$-smoothness also implies $(\bar L_0,L_1)$-smoothness for any $\bar L_0\ge L_0$. By choosing a larger $\bar L_0$, we can switch to Phase II earlier, while the additional cost incurred by using $\bar L_0$ in the accelerated phase is absorbed into the stochastic accuracy term. This yields the following corollary.

\begin{corollary}
\label{cor:end-to-end}
Suppose that Assumptions~\ref{asm: convexity}, \ref{asm: smoothness}, and \ref{asm:oracle} hold, with $L_0,L_1>0$. Let $\|x_0-x^\star\|\le R_0$, where $R_0>0$. For any $\varepsilon>0$ and $\delta\in(0,1)$, a two-phase procedure consisting of \algname{Clip-SGD}, with parameters chosen as in~\cite[Theorem~2]{kadi2026highprobabilityconvergenceclippedsgd}, followed by a single run of \algname{Clip-SSTM} with the calibration specified in the proof, returns a point $y$ satisfying
$
f(y)-f^\star\le\varepsilon
$
with probability at least $1-\delta$, using $
\widetilde{\mathcal O}\left(
1+(L_1R_0)^2
+\sqrt{\frac{L_0R_0^2}{\varepsilon}}
+\left(\frac{\sigma R_0}{\varepsilon}\right)^{\frac{\alpha}{\alpha-1}}
\right)$
stochastic-gradient evaluations.
\end{corollary}

\noindent\textbf{Proof overview of \cref{thm:main}.}\par
\noindent\textbf{1) A simplified accelerated potential.}
Let $g_i:=\operatorname{clip}(\nabla f(x_i,\xi_i),\lambda_i)$ and
$\theta_i:=g_i-\nabla f(x_i)$, so that $\theta_i$ is the error induced
by replacing the true gradient with the clipped stochastic gradient.
For any comparator $u\in\mathbb R^d$, define
$$
\mathcal R_t(u):= \sum_{i=1}^t\alpha_i\langle\theta_i,u-z_{i-1}\rangle +\sum_{i=1}^t\alpha_i^2\norm{\theta_i}^2.
$$
Compared with the stochastic error term in the comparator inequality
of~\citet[Lemma~F.1]{sadiev2023high}, our residual $\mathcal R_t(u)$
has a simpler form: it contains only the clipping error $\theta_i$ and
no additional term involving $\nabla f(x_i)$.

More importantly, the accelerated potential argument does not require
a global classical smoothness bound. It is enough that, at every
iteration, the particular pair $(x_i,y_i)$ satisfies
\begin{equation}
\label{eq:local-descent-overview}
f(y_i) \le f(x_i) + \langle\nabla f(x_i),y_i-x_i\rangle + 4L_0\norm{y_i-x_i}^2.
\end{equation}
Whenever~\cref{eq:local-descent-overview} holds up to time $t$,
\cref{lem:comparator}, together with the weight condition
$A_i\ge 16L_0\alpha_i^2$ satisfied by our standard
\algname{Clip-SSTM} weights, gives, for every $u\in\mathbb R^d$,
\begin{equation}
\label{eq:comparator-overview}
A_t\bigl(f(y_t)-f(u)\bigr) + \frac12\norm{z_t-u}^2 \le \frac12\norm{\bar x-u}^2 + \mathcal R_t(u).
\end{equation}
Thus, rather than replacing $(L_0,L_1)$-smoothness by a classical
smoothness constant on an entire localization ball, our proof only
needs to maintain~\cref{eq:local-descent-overview} along the
accelerated trajectory. This naturally leads to an induction:
assuming the descent inequality holds up to time $t$, we use
\cref{eq:comparator-overview} to obtain bounds on the objective gap
and the auxiliary iterates, and then use these bounds to establish the
descent inequality at iteration $t+1$. The main difficulty is therefore
to control the accumulated stochastic error strongly enough to close
this induction.

\textbf{2) Why two potentials are needed.}
Recall that $\norm{\bar x-x^\star}\le R$. A natural first attempt is to
apply~\cref{eq:comparator-overview} only with $u=x^\star$. If
$\mathcal R_t(x^\star)\le R^2/2$, then
$A_t\bigl(f(y_t)-f^\star\bigr)+\frac12\norm{z_t-x^\star}^2\le R^2$.
Hence,
$$
f(y_t)-f^\star \le \frac{R^2}{A_t} \quad\text{and}\quad \norm{z_t-x^\star} \le \sqrt{2}\,R.
$$
Since $A_t$ grows quadratically, the first bound already has the
desired accelerated form, while the second controls the distance
$\norm{x^\star-z_{i-1}}$ appearing in the stochastic error
$\mathcal R_t(x^\star)$.

However, these bounds alone are not sufficient to close the induction
described above. To recover~\cref{eq:local-descent-overview} at
iteration $t+1$, we first need to control $\norm{\nabla f(y_t)}$ and
then transfer this control from $y_t$ to the next query point
$x_{t+1}$. In our argument, this gradient control is obtained by
keeping the objective gap $f(y_t)-f^\star$ at the scale
$L_0/L_1^2$. When $A_t$ is small, however, the bound $R^2/A_t$ can be
much larger than this scale.

This issue is already visible at initialization. Phase~I guarantees
$
f(\bar x)-f^\star \le \frac{cL_0}{L_1^2},$
but since $A_0=0$, the $x^\star$-potential starts from
$\frac12\norm{\bar x-x^\star}^2$ and does not use this small objective
gap.

To address this difficulty, we introduce a second potential centered at
the switching point $\bar x$. Applying~\cref{eq:comparator-overview}
with $u=\bar x$, and using $z_0=\bar x$, gives
$$
A_t\bigl(f(y_t)-f(\bar x)\bigr) + \frac12\norm{z_t-\bar x}^2 \le \mathcal R_t(\bar x).
$$
Therefore, if $\mathcal R_t(\bar x)\le cA_tL_0/L_1^2$, then
$$
\begin{cases}
f(y_t)-f^\star \le f(\bar x)-f^\star + \dfrac{\mathcal R_t(\bar x)}{A_t} \le \dfrac{2cL_0}{L_1^2},\\[2mm]
\norm{z_t-\bar x} \le \dfrac{2\sqrt{cA_tL_0}}{L_1}.
\end{cases}
$$
The first bound keeps $y_t$ in the small-gap regime and therefore
controls $\norm{\nabla f(y_t)}$ at the scale $L_0/L_1$. Applying the
second bound up to time $t$ controls the auxiliary iterates relative
to $\bar x$. Since $y_t=A_t^{-1}\sum_{j=1}^t\alpha_jz_j$, this gives
$\norm{z_t-y_t}\le 4\sqrt{cA_tL_0}/L_1$. Moreover,
$x_{t+1}-y_t=\frac{\alpha_{t+1}}{A_{t+1}}(z_t-y_t)$, so the weight
conditions ensure that $x_{t+1}$ remains sufficiently close to $y_t$.
This allows the gradient control to be transferred from $y_t$ to
$x_{t+1}$; in particular, $\norm{\nabla f(x_{t+1})}\le L_0/L_1$.

Since the clipped update also keeps $y_{t+1}$ sufficiently close to
$x_{t+1}$, $(L_0,L_1)$-smoothness then gives
\cref{eq:local-descent-overview} at iteration $t+1$. This closes the
deterministic part of the induction.

The two potentials therefore play complementary roles. The
$\bar x$-potential propagates the small-gap condition obtained in
Phase~I, while the $x^\star$-potential provides the accelerated bound
and controls the distance to the minimizer. Together, they give
$$
f(y_t)-f^\star \le \min\left\{\frac{2cL_0}{L_1^2},\frac{R^2}{A_t}\right\},
$$
as well as bounds on both $\norm{z_t-\bar x}$ and
$\norm{z_t-x^\star}$.

Combining these bounds with $(L_0,L_1)$-smoothness and the weight
conditions gives two complementary bounds on the gradient at
$x_{t+1}$, where the next stochastic gradient is evaluated. The
$\bar x$-potential gives a bound of order $\sqrt{c}\,L_0/L_1$, while the $x^\star$-potential gives a bound that decays as $A_{t+1}^{-1/2}$. More precisely, up to numerical constants,
$$
\begin{cases}
\norm{\nabla f(x_{t+1})} \lesssim \dfrac{\sqrt{c}\,L_0}{L_1},\\[2mm]
\norm{\nabla f(x_{t+1})} \lesssim R\sqrt{\dfrac{L_0}{A_{t+1}}}.
\end{cases}
$$
Since both bounds hold simultaneously, we can use the smaller one.
Together with $A_{t+1}\gtrsim \ell^2L_0\alpha_{t+1}^2$, this suggests
the clipping scale $
\frac{1}{\alpha_{t+1}\ell} \min\left\{ \frac{\sqrt{cA_{t+1}L_0}}{L_1},R\right\}.$
Accordingly, for every $i=1,\ldots,K$, we choose
$$
\lambda_i:= \frac{1}{\alpha_i\ell}\min\left\{ \frac{\sqrt{cA_iL_0}}{L_1},R
\right\}.
$$
With the exact constants in
\cref{lem:gradient-below-clipping}, whenever the induction conditions
hold up to time $t$, this choice gives
$$
\norm{\nabla f(x_{t+1})}\le\frac{\lambda_{t+1}}{2}.
$$
Thus, before evaluating the stochastic gradient at $x_{t+1}$, the true
gradient is guaranteed to lie below half the clipping threshold. This
is precisely the condition needed for the clipping-bias and
second-moment bounds used in the stochastic analysis.

\textbf{3) Controlling the accumulated stochastic error.}
It remains to show that the residual bounds used above hold
simultaneously throughout the run. The difficulty is that the clipping
estimates require $\norm{\nabla f(x_i)}\le\lambda_i/2$, while this
gradient bound is itself obtained only as long as the preceding descent
and residual bounds remain valid. We handle this dependence with a
stopped process.

For $t=0,\ldots,K$, let $\mathcal E_t$ be the event that, for every
$i\le t$, the descent inequality~\cref{eq:local-descent-overview} holds
and
$$
\mathcal R_i(\bar x)\le\frac{cA_iL_0}{L_1^2}
\quad\text{and}\quad
\mathcal R_i(x^\star)\le\frac{R^2}{2}.
$$
For $t=0$ there are no conditions, so $\mathbf 1_{\mathcal E_0}=1$.
For $u\in\{\bar x,x^\star\}$, define the stopped residual
$$
\widehat{\mathcal R}_t(u):=\sum_{i=1}^t\mathbf 1_{\mathcal E_{i-1}}\left(\alpha_i\langle\theta_i,u-z_{i-1}\rangle+\alpha_i^2\norm{\theta_i}^2\right).
$$
The indicator ensures that each summand is analyzed only when the
previous induction conditions hold. In particular, on
$\mathcal E_{i-1}$, \cref{lem:gradient-below-clipping} gives
$\norm{\nabla f(x_i)}\le\lambda_i/2$, and hence the following clipping estimates hold:
$$
\begin{cases}
\norm{\E[\theta_i\mid\mathcal F_{i-1}]} \le 
\dfrac{2^\alpha\sigma^\alpha}{\lambda_i^{\alpha-1}},\\[2mm]
\E[\norm{\theta_i}^2\mid\mathcal F_{i-1}] \le 2\sigma^\alpha\lambda_i^{2-\alpha}.
\end{cases}
$$

We then subtract the conditional means and write
$\widehat{\mathcal R}_t(u)=M_t(u)+B_t(u)$, where $M_t(u)$ is a
martingale and $B_t(u)$ collects the clipping bias and conditional
second-moment terms. The localization bounds obtained from the two
potentials control the comparator distances appearing in both terms. In particular,
$$
\begin{cases}
\norm{\bar x-z_{i-1}}\lesssim s_i,\\[2mm]
\norm{x^\star-z_{i-1}}\lesssim R,
\end{cases}
\qquad
s_i:=\min\left\{\frac{\sqrt{cA_iL_0}}{L_1},R\right\}.
$$
Since $\alpha_i\lambda_i=s_i/\ell$, both the predictable terms and the
conditional variances can be controlled by the same quantity
$
S_t:=\sigma^\alpha\sum_{i=1}^t\frac{\alpha_i}{s_i\lambda_i^{\alpha-1}}.$

Thus, if $S_K\le 1/512$, the predictable part is sufficiently small,
while the martingale has controlled conditional variance and bounded
increments. Applying Freedman's inequality and taking a union bound over
$t=1,\ldots,K$ and $u\in\{\bar x,x^\star\}$ yields, with probability
at least $1-\delta$, simultaneously for every $t$,
$$
\begin{cases}
\widehat{\mathcal R}_t(\bar x)\le\dfrac{cA_tL_0}{L_1^2},\\[2mm]
\widehat{\mathcal R}_t(x^\star)\le\dfrac{R^2}{2}.
\end{cases}
$$

It remains to transfer these bounds from the stopped process to the
original residuals. Suppose that $\mathcal E_{t-1}$ holds. Then
$\widehat{\mathcal R}_t(u)=\mathcal R_t(u)$ for
$u\in\{\bar x,x^\star\}$. Moreover, the induction hypotheses together
with the weight and clipping-level conditions satisfy the assumptions
of \cref{lem:grad_bound_smoothness} with index $t-1$. Hence, \cref{eq:local-descent-overview} holds at iteration $t$. The two
residual bounds above therefore imply that $\mathcal E_t$ holds. Starting from
$\mathbf 1_{\mathcal E_0}=1$ and proceeding by induction, we obtain
$\mathcal E_K$ with probability at least $1-\delta$. Therefore, all the
descent and localization bounds established above hold simultaneously
throughout the run.

\textbf{4) Verifying the parameter choice and obtaining the rate.}
It remains to check that the parameters in \cref{thm:main} satisfy
$S_K\le 1/512$. Using $\lambda_i=s_i/(\alpha_i\ell)$ and the explicit
accelerated weights, $S_K$ splits into two contributions, corresponding
to the two terms in the definition of $s_i$. By the choice of $a$, each
contribution is at most $1/1024$. Hence $S_K\le 1/512$, and therefore
$\mathcal E_K$ holds with probability at least $1-\delta$.

On this event, the two potentials give
$$
f(y_K)-f^\star \le \min\left\{ \frac{2cL_0}{L_1^2},\frac{R^2}{A_K}\right\}.
$$
Since $A_K=K(K+3)/(2aL_0)$, substituting the definition of $a$ gives,
up to numerical constants,
$$
\frac{R^2}{A_K} \lesssim \frac{L_0R^2\ell^2}{K^2} + \frac{\sigma^2L_1^2R^2}{cL_0}
\frac{\ell^{2(\alpha-1)/\alpha}}{K^{2(\alpha-1)/\alpha}} + \frac{\sigma R\,\ell^{(\alpha-1)/\alpha}}
{K^{(\alpha-1)/\alpha}}.
$$

The middle term contains $L_1$, but combining it with the low-gap bound
$2cL_0/L_1^2$ and using $\min\{x,y\}\le\sqrt{xy}$ gives another term of
order $\frac{\sigma R\,\ell^{(\alpha-1)/\alpha}}{K^{(\alpha-1)/\alpha}}.$ Therefore,
$$
f(y_K)-f^\star =\widetilde{\mathcal O}\left(\frac{L_0R^2}{K^2}+\frac{\sigma R}{K^{(\alpha-1)/\alpha}}
\right).
$$
Thus, after the warm-start condition $f(\bar x)-f^\star\le cL_0/L_1^2$ has been reached, the accelerated Phase~II rate no longer depends on $L_1$. The $L_1$-dependence of the
two-phase method comes only from the Phase~I cost required to reach this warm-start regime.

\subsection{One-phase result}
\label{sec:one-phase-result}

We next analyze the one-phase procedure introduced in
Section~\ref{sec:prel}. In contrast to the two-phase
procedure, the method starts directly from $x_0$ and does not
require a small initial objective gap. The modified weight schedule
first controls the high-curvature transient and then recovers the
usual accelerated regime.

\begin{theorem}
\label{thm:one-phase}
Suppose that Assumptions~\ref{asm: convexity}, \ref{asm: smoothness}, and
\ref{asm:oracle} hold, where $L_0,L_1>0$.
Assume that $\norm{x_0-x^\star} \le R_0$, where $R_0>0$, and let
$\bar\Delta_0\ge f(x_0)-f^\star$ be a deterministic upper bound.

Let $K\ge1$, $\delta\in(0,1)$, and set $\ell:=8\log\frac{4K}{\delta},$ and
$
a:=\max\left\{128\ell^2,\;512^{1/\alpha}\frac{\sigma}{L_0R_0}\ell^{\frac{\alpha-1}{\alpha}}K^{\frac{\alpha+1}{\alpha}}\right\}.$
Choose $A_0>0$ satisfying $A_0\le\frac{1}{aL_0}$ and $A_0\bar\Delta_0\le\frac{R_0^2}{64}.$
For every $k\ge0$, let
$$
\alpha_{k+1}:=\min\left\{\dfrac{A_k}{64\ell(1+L_1R_0)},\sqrt{\dfrac{A_k}{aL_0}}\right\},
$$
and, for $i=1,\ldots,K$, let $\lambda_i:=\frac{R_0}{\alpha_i\ell}.$ Define
$
K_0:=\left\lceil128\ell(1+L_1R_0)\log\left(\frac{[64\ell(1+L_1R_0)]^2}{aL_0A_0}\right)\right\rceil.
$
Then, if $K\ge2K_0$, with probability at least $1-\delta$, the output $y_K$ of \algname{Clip-SSTM} satisfies
$$
f(y_K)-f^\star=\widetilde{\mathcal O}\left(\frac{L_0R_0^2}{K^2}+\frac{\sigma R_0}{K^{(\alpha-1)/\alpha}}\right),
$$
where $\widetilde{\mathcal O}$ suppresses logarithmic factors in $K$ and $1/\delta$. With the initialization below, the dependence on $\bar\Delta_0$
and $L_1$ enters only through $K_0$, which is accuracy-independent up to logarithmic factors.
\end{theorem}
\begin{remark}\label{rem_init}
The initialization parameter $A_0$ can be chosen explicitly as $
A_0:=\min\left\{\frac{1}{aL_0},\frac{R_0^2}{64\bar\Delta_0}\right\},$
with the convention that $R_0^2/(64\bar\Delta_0)=+\infty$ when
$\bar\Delta_0=0$. This choice guarantees
$A_0\le 1/(aL_0)$ and $A_0\bar\Delta_0\le R_0^2/64$. Since
$\frac{1}{aL_0A_0}=\max\left\{1,\frac{64\bar\Delta_0}{aL_0R_0^2}\right\}$,
the transient becomes
$$K_0=\Bigg\lceil 128\ell(1+L_1R_0)\log\Big([64\ell(1+L_1R_0)]^2\max\left\{1,\frac{64\bar\Delta_0}{aL_0R_0^2}\right\}\Big)\Bigg\rceil .$$
Consequently, $K_0=\widetilde{\mathcal O}(1+L_1R_0)$,
where the notation suppresses logarithmic factors in
$K$, $1/\delta$, $1+L_1R_0$, and
$1+\bar\Delta_0/(L_0R_0^2)$.
The transient is therefore accuracy-independent up to
logarithmic factors: its dependence on the target accuracy
through the horizon $K$ is at most logarithmic.
\end{remark}

\begin{remark}
In the worst case, the logarithmic term involving $\bar\Delta_0$
can add a factor of order $L_1R_0$ to the transient.
Indeed, integrating the gradient bound in
\cref{lem:gradient-growth} along the segment from $x^\star$
to $x_0$ gives
$$
\Delta_0
\le \frac{L_0}{L_1^2}
\left(e^{L_1R_0}-1-L_1R_0\right).
$$
We can use this upper bound as $\bar\Delta_0$.
It depends only on $L_0$, $L_1$, and $R_0$, so we do not need
to know $f^\star$.
With this choice, the initialization above gives
$K_0=\widetilde{\mathcal O}((1+L_1R_0)^2)$,
where the notation hides logarithmic factors in $K$,
$1/\delta$, and $1+L_1R_0$.
If a tighter bound on the initial gap is available,
the transient can be smaller.
\end{remark}

\paragraph{Proof overview of \cref{thm:one-phase}.}
We prove by induction that the local descent inequality and a
bound on the accumulated error hold throughout the run.
Let $\Delta_t:=f(y_t)-f^\star$, $\theta_i:=g_i-\nabla f(x_i)$, and
$\mathcal R_t:=\sum_{i=1}^t\alpha_i\langle\theta_i,x^\star-z_{i-1}\rangle
+\sum_{i=1}^t\alpha_i^2\|\theta_i\|^2$.
For $t=0,\ldots,K$, define
$$
\mathcal E_t:=
\bigcap_{i=1}^t
\left\{
\begin{aligned}
f(y_i)&\le f(x_i)+\langle\nabla f(x_i),y_i-x_i\rangle
+\frac{A_i}{4\alpha_i^2}\|y_i-x_i\|^2,\\
\mathcal R_i&\le \frac{R_0^2}{2}
\end{aligned}
\right\},
$$
with $\mathcal E_0$ the whole sample space.
Our goal is to prove $\mathbb P(\mathcal E_K)\ge1-\delta$.

Fix $0\le t<K$ and suppose $\mathcal E_t$ holds.
By \cref{lem:comparator2} and the initialization
$A_0\Delta_0\le R_0^2/64$, for every $i=1,\ldots,t$,
$$
A_i\Delta_i+\frac12\|z_i-x^\star\|^2
\le A_0\Delta_0+\frac12\|x_0-x^\star\|^2+\mathcal R_i
\le \frac{65}{64}R_0^2.
$$
Thus $\Delta_i\le2R_0^2/A_i$ and
$\|z_i-x^\star\|\le3R_0/2$.
These bounds also hold at $i=0$ by initialization.
Since $y_t$ is a convex combination of $x_0,z_1,\ldots,z_t$,
we obtain $\|z_t-y_t\|\le3R_0$.

We next establish the descent inequality at iteration $t+1$.
The gradient--gap bound in \cref{lem:self-bounding-norm-form}
controls $\|\nabla f(y_t)\|$ using $\Delta_t\le2R_0^2/A_t$.
The identity
$x_{t+1}-y_t=(\alpha_{t+1}/A_{t+1})(z_t-y_t)$
and the weight restriction
$\alpha_{t+1}/A_t\le[64\ell(1+L_1R_0)]^{-1}$
keep $x_{t+1}$ close to $y_t$.
Hence \cref{lem:gradient-growth}, together with
$A_t\ge aL_0\alpha_{t+1}^2$, gives
$L_0+L_1\|\nabla f(x_{t+1})\|
\le A_{t+1}/(8\alpha_{t+1}^2)$.
Clipping also bounds the distance from $x_{t+1}$ to $y_{t+1}$.
These two bounds give the next descent inequality
by \cref{lem:one-phase-descent}.
Under $a\ge128\ell^2$, \cref{lem:one-phase-gradient} further gives
$\|\nabla f(x_{t+1})\|\le\lambda_{t+1}/2$,
where $\lambda_i=R_0/(\alpha_i\ell)$.

To control the accumulated error, define
$\widehat{\mathcal R}_t:=
\sum_{i=1}^t\mathbf 1_{\mathcal E_{i-1}}
\bigl(\alpha_i\langle\theta_i,x^\star-z_{i-1}\rangle
+\alpha_i^2\|\theta_i\|^2\bigr)$.
Each term is included only when the induction conditions hold
through iteration $i-1$.
On $\mathcal E_{i-1}$, the preceding argument gives
$\|\nabla f(x_i)\|\le\lambda_i/2$, so the conditional clipping-bias
and second-moment bounds apply.
Using Freedman's inequality (\cref{cor:freedman-hp}),
\cref{lem:one-phase-residual} gives an event $\mathcal G$ with
$\mathbb P(\mathcal G)\ge1-\delta$ on which
$\widehat{\mathcal R}_t\le R_0^2/2$ for every $t\le K$, provided
$\sigma^\alpha\ell^{\alpha-1}R_0^{-\alpha}
\sum_{i=1}^K\alpha_i^\alpha\le1/512$.
By \cref{lem:one-phase-weights},
$\sum_{i=1}^K\alpha_i^\alpha\le K^{\alpha+1}/(aL_0)^\alpha$,
and the chosen value of $a$ ensures this condition.

We now complete the induction on $\mathcal G$.
The event $\mathcal E_0$ holds by definition.
Suppose $\mathcal E_t$ holds for some $0\le t<K$.
Since the events $\mathcal E_i$ are nested, every indicator
in $\widehat{\mathcal R}_{t+1}$ equals one.
Therefore, on $\mathcal G$,
$\mathcal R_{t+1}=\widehat{\mathcal R}_{t+1}\le R_0^2/2$.
The preceding deterministic argument also gives the descent
inequality at iteration $t+1$, so $\mathcal E_{t+1}$ holds.
Thus $\mathcal G\subseteq\mathcal E_K$, proving
$\mathbb P(\mathcal E_K)\ge1-\delta$.

Finally, set $b:=64\ell(1+L_1R_0)$.
While $A_t<b^2/(aL_0)$, the weights give
$A_{t+1}=(1+1/b)A_t$.
After this threshold is reached, the square-root branch remains
active and $A_t$ grows quadratically.
By \cref{lem:one-phase-weights},
$A_K\ge4(K-K_0)^2/(25aL_0)$ for $K\ge K_0$.
On $\mathcal E_K$, the potential bound gives
$\Delta_K\le2R_0^2/A_K$.
Hence, for $K\ge2K_0$, the choice of $a$ yields
$$
f(y_K)-f^\star
\lesssim \frac{aL_0R_0^2}{K^2}
\lesssim
\frac{L_0R_0^2\ell^2}{K^2}
+\frac{\sigma R_0\ell^{(\alpha-1)/\alpha}}
{K^{(\alpha-1)/\alpha}},
$$
with probability at least $1-\delta$.
With the initialization in \cref{rem_init}, $L_1$ appears only
in the transient $K_0$, which is accuracy-independent up to
logarithmic factors.
\section{Related Work}
\label{sec:related-work}

\paragraph{Optimization under generalized smoothness.}
The $(L_0,L_1)$-smoothness model captures objectives whose
local curvature can grow with the gradient norm, a dependence highlighted by
\citet{zhang2019gradient}. A first-order formulation was introduced by
\citet{zhang2020improved}. For convex problems, clipping and related first-order methods have been
studied in
\citep{koloskova2023revisiting,gorbunov2025methods,
vankov2025optimizing,lobanov2024linear}.
In particular, recent works
\citep{gorbunov2025methods,vankov2025optimizing,lobanov2024linear}
obtain nonaccelerated guarantees without exponential dependence on $L_1R$.
Deterministic acceleration has also been investigated.
\citet{li2023convex} established accelerated convergence
under an $\ell$-smoothness condition with subquadratic $\ell$,
with a rate coefficient depending on an effective smoothness
bound determined by the initialization. \citet{gorbunov2025methods} proposed an accelerated Similar Triangles method whose guarantee retains an exponential dependence on $L_1R$, whereas
\citet{vankov2025optimizing} removed this dependence using a
one-dimensional line search with function values.
\citet{pmlr-v306-tyurin26b} obtained accelerated guarantees
under $\ell$-smoothness with a leading high-accuracy term
depending on $\ell(0)$ and an accuracy-independent transient.

\paragraph{Stochastic optimization under generalized smoothness.}
Several works study stochastic methods beyond classical smoothness
\citep{zhang2020adaptive,chen2023generalized,li2023convergence}.
For convex $(L_0,L_1)$-smooth objectives,
\citet{gaash2025convergence} establish high-probability guarantees under
sub-Gaussian noise, while
\citet{kadi2026highprobabilityconvergenceclippedsgd} obtain
high-probability guarantees for standard \algname{Clip-SGD} under bounded central
$\alpha$-th moment noise, $\alpha\in(1,2]$, without exponential dependence
on $L_1R$.
\citet{yu2025convergence} analyze stochastic accelerated-gradient methods
under relaxed affine variance assumptions, but their convex guarantee does
not achieve an accelerated rate.
Among stochastic accelerated results, \citet{dvinskikh2026localize} obtain high-probability acceleration under norm-sub-Gaussian noise without exponential dependence
on $L_1R$. However, their convex oracle complexity retains
multiplicative factors $(1+L_1\bar R)^{1/2}$ and
$(1+L_1\bar R)^3$ in the optimization and statistical terms,
respectively, where $\bar R$ is their localization radius.
Our guarantees allow heavy-tailed noise and, up to logarithmic
factors, have $L_1$-independent leading high-accuracy terms.

\paragraph{High-probability methods under heavy-tailed noise.}
Gradient clipping has become a standard tool for high-probability
optimization under heavy-tailed stochastic gradients.
Early work includes truncated stochastic mirror descent
\citep{nazin2019algorithms} and high-probability analyses of \algname{Clip-SGD} and
stochastic Similar Triangles methods
\citep{gorbunov2020stochastic}.
Subsequent work established high-probability guarantees for
clipping-based stochastic methods under bounded central
$\alpha$-th moment assumptions with $\alpha<2$
\citep{cutkosky2021high,sadiev2023high,nguyen2023improved,
chezhegov2024gradient}.
In particular, under classical $L$-smoothness, accelerated
high-probability rates under heavy-tailed noise have been established in
\citep{sadiev2023high,nguyen2023improved}.
Our results connect this line of work with generalized smoothness:
under $(L_0,L_1)$-smoothness and bounded central $\alpha$-th moment noise,
$\alpha\in(1,2]$, we recover, up to logarithmic factors, the classical
accelerated rate while avoiding exponential dependence on $L_1R$.

\paragraph{Concurrent work.} \label{par:concurrent-work}
Concurrently, \citet{lobanov2026acceleratedstochasticmethodh0h1smoothness}
established high-probability accelerated guarantees under
$(H_0,H_1)$-smoothness and a more general heavy-tailed noise model,
allowing both gradient-dependent and objective-gap-dependent noise
components. Their method combines clipping with projections and
phase restarts. Our approaches use different mechanisms: a warm
start followed by a single accelerated run, or a single run with
modified deterministic weights. Our results show that, under
$(L_0,L_1)$-smoothness and a uniform bounded central $\alpha$-th
noise moment, neither projections nor repeated restarts are
necessary to recover the classical accelerated rate up to
logarithmic factors, with $L_1$ confined to an initial transient.

\section{Conclusion}

We established high-probability accelerated convergence guarantees for convex $(L_0,L_1)$-smooth optimization with unbiased stochastic gradients under heavy-tailed noise. Our results show that the exponential dependence on $L_1R$ arising from direct localization is not inherent. In the two-phase approach, \algname{Clip-SGD} provides a warm start in a low-curvature region, after which \algname{Clip-SSTM} attains the classical accelerated rate up to logarithmic factors:
$$
\widetilde{\mathcal O}\!\left(\frac{L_0R^2}{K^2}+\frac{\sigma R}{K^{(\alpha-1)/\alpha}}\right).
$$
A warm start is not necessary: a modified weight schedule allows \algname{Clip-SSTM} to start directly from the initial point, control the high-curvature transient, and recover the same accelerated regime. With the prescribed initialization, both approaches confine the $L_1$-dependence to an accuracy-independent transient, up to logarithmic factors.

These results show that acceleration is compatible with generalized smoothness and heavy-tailed noise, and that classical clipped accelerated schemes retain their high-accuracy behavior beyond global smoothness. A natural direction is to develop adaptive variants requiring less prior knowledge of problem parameters while preserving these high-probability guarantees.

\bibliographystyle{apalike}
\bibliography{conference}

\clearpage

\appendix

\section{Notation, Useful Identities, and Preliminary Results}
\label{sec:prelim-results}

\paragraph{Notation.}

\Cref{tab:notation} summarizes the main notation for the problem,
initialization, algorithm, and stochastic analysis.
Most notation is shared by the two-phase and one-phase methods.
When definitions differ, as for the filtration $\mathcal F_t$,
the table describes their role; precise definitions are given
in the corresponding analysis.
\begin{table}[H]
\centering
\caption{Main notation. Quantities whose definitions differ between
the two procedures are specified in the corresponding analysis.}
\label{tab:notation}
\renewcommand{\arraystretch}{1.25}
\begin{tabular}{@{}p{0.20\textwidth}p{0.75\textwidth}@{}}
\toprule
Notation & Meaning \\
\midrule
\multicolumn{2}{@{}l}{\textbf{Problem and initialization}} \\
$\Delta(x)$
& Objective gap $f(x)-f^\star$. \\
$\delta$
& Target failure probability. \\
$x_0$, $R_0$
& Initial point and initial distance bound
$\|x_0-x^\star\|\le R_0$ in both methods. \\
$\bar x$
& Output of Phase I and starting point of Phase II. \\
$v_0$
& Starting point of \algname{Clip-SSTM}:
$v_0=\bar x$ in Phase II and $v_0=x_0$ in the one-phase procedure. \\
$R$
& Distance bound $\|\bar x-x^\star\|\le R$
at the start of Phase II. \\
$\Delta_0$, $\bar\Delta_0$
& Initial gap $\Delta_0=f(x_0)-f^\star$ and a deterministic
upper bound $\bar\Delta_0\ge\Delta_0$ used in the one-phase procedure. \\
\midrule
\multicolumn{2}{@{}l}{\textbf{Algorithm and parameters}} \\
$K$
& Number of \algname{Clip-SSTM} iterations. \\
$x_t,y_t,z_t$
& Gradient query point, output iterate, and auxiliary iterate. \\
$\alpha_t$, $A_t$
& Positive iteration weight and accumulated weight
$A_t=A_0+\sum_{i=1}^{t}\alpha_i$. \\
$\lambda_t$
& Clipping threshold at iteration $t$. \\
$a$, $\ell$
& Weight-scaling and confidence parameters,
chosen in the corresponding theorem. \\
$c$
& Constant specifying the small-gap threshold $cL_0/L_1^2$
in the two-phase analysis. \\
$K_0$
& Bound on the initial transient in the one-phase analysis. \\
\midrule
\multicolumn{2}{@{}l}{\textbf{Stochastic analysis}} \\
$G_t$, $g_t$
& Stochastic gradient $G_t=\nabla f(x_t,\xi_t)$ and its
clipped version $g_t=\operatorname{clip}(G_t,\lambda_t)$. \\
$\theta_t$
& Error relative to the true gradient:
$\theta_t=g_t-\nabla f(x_t)$. \\
$\mathcal F_t$
& History after $t$ \algname{Clip-SSTM} iterations;
$\mathcal F_0$ contains the Phase-I history in the two-phase
analysis and is trivial in the one-phase analysis. \\
\bottomrule
\end{tabular}
\end{table}
\clearpage
\paragraph{Useful identities for \algname{Clip-SSTM}.}
We collect in \cref{tab:sstm-identities} the identities used
throughout the proofs. They follow directly from the updates of
Clip-SSTM and hold for both weight schedules.
Here $y_0=z_0=v_0$ and
$g_{t+1}:=\operatorname{clip}
(\nabla f(x_{t+1},\xi_{t+1}),\lambda_{t+1})$.

\begin{table}[ht]
\centering
\caption{Useful identities for \algname{Clip-SSTM}.}
\label{tab:sstm-identities}
\renewcommand{\arraystretch}{2.2}
\begin{tabular}{@{}ll@{}}
\toprule
Quantity & Identity \\
\midrule
Accumulated weight
&
$\displaystyle
A_t=A_0+\sum_{i=1}^{t}\alpha_i
$
\\
Auxiliary update
&
$\displaystyle
z_{t+1}-z_t=-\alpha_{t+1}g_{t+1}
$
\\
Query displacement
&
$\displaystyle
x_{t+1}-y_t
=\frac{\alpha_{t+1}}{A_{t+1}}(z_t-y_t)
$
\\
Output displacement
&
$\displaystyle
y_{t+1}-x_{t+1}
=\frac{\alpha_{t+1}}{A_{t+1}}(z_{t+1}-z_t)
=-\frac{\alpha_{t+1}^{2}}{A_{t+1}}g_{t+1}
$
\\
Comparator identity
&
$\displaystyle
A_t y_t+\alpha_{t+1}u-A_{t+1}x_{t+1}
=\alpha_{t+1}(u-z_t)
$
\\
Averaging identity
&
$\displaystyle
y_t=\frac{A_0}{A_t}v_0
+\frac{1}{A_t}\sum_{i=1}^{t}\alpha_i z_i,
\qquad \text{if } A_t\neq 0
$
\\
Squared-distance identity
&
$\displaystyle
\frac12\|z_{t+1}-u\|^2-\frac12\|z_t-u\|^2
=\alpha_{t+1}\langle g_{t+1},u-z_t\rangle
+\frac{\alpha_{t+1}^{2}}{2}\|g_{t+1}\|^2
$
\\
\bottomrule
\end{tabular}
\end{table}

\begin{proposition}[{\cite[Lemma~2.2]{gorbunov2025methods}}]
\label{prop:gorbunov}
Suppose \cref{asm: smoothness} holds.
Then, for all $x\in\mathbb R^d$,
$$
\nu\norm{\nabla f(x)}^2 \le 2\left(L_0+L_1\norm{\nabla f(x)}\right)\Delta(x),
$$
where $\nu$ is the solution of $x\exp(x)=1$.
\end{proposition}

\begin{lemma}[Consequence of
{\cite[Lemma~2.1]{gorbunov2025methods}}]
\label{lem:local-smoothness-short-segment}
Assume that \cref{asm: smoothness} holds.
Let $x,y\in\mathbb R^d$, and set $s:=y-x$.
If $L_1\norm{s}\le1$, then
$$
f(y) \le f(x)+\ip{\nabla f(x)}{s}+2\left(L_0+L_1\norm{\nabla f(x)}\right)\norm{s}^2.
$$
\end{lemma}

\begin{lemma}
\label{lem:self-bounding-norm-form}
Assume that \cref{asm: smoothness} holds.
Then, for every $x\in\mathbb R^d$,
$$
\norm{\nabla f(x)}
\le\frac{L_1\Delta(x)+\sqrt{L_1^2\Delta(x)^2+2\nu L_0\Delta(x)}}{\nu}.
$$
\end{lemma}

\begin{proof}[Proof of \cref{lem:self-bounding-norm-form}]
By \cref{prop:gorbunov},
$$
\nu\norm{\nabla f(x)}^2
-2L_1\Delta(x)\norm{\nabla f(x)}
-2L_0\Delta(x)\le0.
$$
Define the quadratic polynomial
$$
p(u):=\nu u^2-2L_1\Delta(x)u-2L_0\Delta(x).$$
The preceding inequality states that $p(\norm{\nabla f(x)})\le0$.
The roots of $p$ are
$$
u_\pm=\frac{L_1\Delta(x)\pm\sqrt{L_1^2\Delta(x)^2+2\nu L_0\Delta(x)}}{\nu}.
$$
We have $p(u)\le0$ only for $u\in[u_-,u_+]$.
Therefore,
$$
\norm{\nabla f(x)} \le u_+ =\frac{L_1\Delta(x)+\sqrt{L_1^2\Delta(x)^2+2\nu L_0\Delta(x)}}{\nu}.
$$
\end{proof}

\begin{lemma}
\label{lem:gradient-growth}
Assume that \cref{asm: smoothness} holds.
Then, for all $x,y\in\mathbb R^d$,
$$
L_0+L_1\norm{\nabla f(y)}
\le e^{L_1\norm{y-x}}
\bigl(L_0+L_1\norm{\nabla f(x)}\bigr).
$$
\end{lemma}
\begin{proof}[Proof of \cref{lem:gradient-growth}]
Let $n\in\mathbb N$ satisfy $n>L_1\norm{y-x}$ and define
$x_i=x+\frac{i}{n}(y-x)$ for $i=0,\ldots,n$, and
$M_i=\sup_{u\in[x_i,x_{i+1}]}\norm{\nabla f(u)}$.
For every $u\in[x_i,x_{i+1}]$,
$$
\norm{\nabla f(u)} \le \norm{\nabla f(x_i)} +\frac{\norm{y-x}}{n}(L_0+L_1M_i),
$$
hence
$$
L_0+L_1M_i \le L_0+L_1\norm{\nabla f(x_i)} +\frac{L_1\norm{y-x}}{n}(L_0+L_1M_i).
$$
Therefore,
$$
L_0+L_1\norm{\nabla f(x_{i+1})}
\le
\left(1-\frac{L_1\norm{y-x}}{n}\right)^{-1}
\bigl(L_0+L_1\norm{\nabla f(x_i)}\bigr).
$$
Iterating for $i=0,\ldots,n-1$ gives
$$
L_0+L_1\norm{\nabla f(y)} \le \left(1-\frac{L_1\norm{y-x}}{n}\right)^{-n}
\bigl(L_0+L_1\norm{\nabla f(x)}\bigr).
$$
Letting $n\to\infty$ and using
$\left(1-\frac{L_1\norm{y-x}}{n}\right)^{-n} \to e^{L_1\norm{y-x}}$
yields the result.
\end{proof}

\begin{proposition}[{\cite[Corollary 1]{kadi2026highprobabilityconvergenceclippedsgd}}]
\label{cor:freedman-hp}
Let $(X_t)_{t\ge 1}$ be a martingale difference sequence with respect to a
filtration $(\mathcal G_t)_{t\ge0}$ such that
$$
    |X_t| \le c
    \quad\text{almost surely for all } t\ge1.
$$
Define
$V_T :=\sum_{t=1}^T \E\!\left[X_t^2\mid\mathcal G_{t-1}\right].$
Assume that, for some deterministic $F>0$,
$V_T \le F \quad \text{almost surely}.$
Then, for every $\delta\in(0,1)$, with probability at least $1-\delta$,
$$
\left|\sum_{t=1}^T X_t \right| \le \sqrt{ 2F\log\frac{2}{\delta}} + \frac{2c}{3}\log\frac{2}{\delta}.
$$
\end{proposition}

\section{High-Probability Analysis for the Two-Phase Result}
\label{sec:phase-two}

 Let $\mathcal F_0$ denote the Phase-I history and set
$$
\mathcal F_i:=\sigma(\mathcal F_0,\xi_1,\ldots,\xi_i), \qquad i=1,\ldots,K.
$$
Then $x_i$ is $\mathcal F_{i-1}$-measurable. Recall that
$$
G_i:=\nabla f(x_i,\xi_i), \qquad g_i:=\clip(G_i,\lambda_i), \qquad \theta_i:=g_i-\nabla f(x_i).
$$
Since $\xi_i$ is independent of $\mathcal F_{i-1}$,
the conditions in \cref{asm:oracle} hold conditionally
on $\mathcal F_{i-1}$ at $x_i$.

For every $u\in\mathbb R^d$, let
$$
\mathcal R_t(u):= \sum_{i=1}^t\alpha_i\ip{\theta_i}{u-z_{i-1}}+
\sum_{i=1}^t\alpha_i^2\norm{\theta_i}^2.
$$
The following comparator inequality is the deterministic backbone of the analysis.

\begin{lemma}
\label{lem:comparator}
Suppose that \cref{asm: convexity} holds, where $L_0>0$.
Assume that $A_0=0$ and $y_0=z_0=\bar x$. Let $K\ge1$ and suppose that
$$
A_i\ge 16L_0\alpha_i^2, \qquad i=1,\ldots,K.
$$
Let $t\in\{1,\ldots,K\}$ and assume that, for every $i=1,\ldots,t$,
$$
f(y_i)\le f(x_i)+\ip{\nabla f(x_i)}{y_i-x_i}
+4L_0\norm{y_i-x_i}^2.
$$
Then, for every $u\in\mathbb R^d$,
$$
A_t(f(y_t)-f(u))+\frac12\norm{z_t-u}^2
\le
\frac12\norm{\bar x-u}^2+\mathcal R_t(u).
$$
\end{lemma}

\begin{lemma}
\label{lem:grad_bound_smoothness}
Suppose that Assumptions~\ref{asm: convexity} and~\ref{asm: smoothness}
hold with $L_0,L_1>0$. Assume that $f(\bar x)-f^\star\le\frac{cL_0}{L_1^2}$
for some $0<c\le\frac1{32}$. Assume moreover that $A_0=0$ and $y_0=z_0=\bar x$.
Let $t\in\{0,\ldots,K-1\}$ and suppose that for all $i \in \{1,\ldots,t+1\}$, we have
$A_i\ge 16L_0\alpha_i^2.$ Suppose that, for every $i=1,\ldots,t$,
$$
\begin{cases}
\displaystyle f(y_i)\le f(x_i)+\ip{\nabla f(x_i)}{y_i-x_i} +4L_0\norm{y_i-x_i}^2,\\[2mm]
\displaystyle \mathcal R_i(\bar x)\le\frac{cA_iL_0}{L_1^2},
\end{cases}
$$
and that
$$
\lambda_{t+1} \le \frac{A_{t+1}}{L_1\alpha_{t+1}^2}.
$$
Then
$$
\begin{cases}
\displaystyle \norm{\nabla f(x_{t+1})}\le\frac{L_0}{L_1},\\[2mm]
\displaystyle L_1\norm{y_{t+1}-x_{t+1}}\le1.
\end{cases}
$$
Consequently,
$$
f(y_{t+1}) \le f(x_{t+1}) +\ip{\nabla f(x_{t+1})}{y_{t+1}-x_{t+1}} +4L_0\norm{y_{t+1}-x_{t+1}}^2.
$$
\end{lemma}
\begin{lemma}
\label{lem:gradient-below-clipping}
Suppose that \cref{asm: convexity,asm: smoothness} hold with $L_0,L_1>0$, and assume that
$
\norm{\bar x-x^\star}\le R$ and $f(\bar x)-f^\star\le\frac{cL_0}{L_1^2}$, where $R>0$ and $0<c\le1/32$. Let $\ell\ge 4$ and assume that $A_0=0$ and $y_0=z_0=\bar x$. Suppose that
$$
\begin{cases}
\displaystyle A_i\ge 96\ell^2L_0\alpha_i^2,
& i=1,\ldots,K,\\[2mm]
\displaystyle A_{i+1}\le\frac52A_i,
& i=1,\ldots,K-1.
\end{cases}
$$
Define
$$
\lambda_i:=\frac1{\alpha_i\ell} \min\left\{\frac{\sqrt{cA_iL_0}}{L_1},R\right\},
\qquad i=1,\ldots,K.
$$
Let $t\in\{0,\ldots,K-1\}$ and suppose that
$$
\begin{cases}
\displaystyle f(y_i)\le f(x_i)+\ip{\nabla f(x_i)}{y_i-x_i}+4L_0\norm{y_i-x_i}^2,\\[2mm]
\displaystyle \mathcal R_i(\bar x)\le\frac{cA_iL_0}{L_1^2},\\[2mm]
\displaystyle \mathcal R_i(x^\star)\le\frac{R^2}{2},
\end{cases}
\qquad i=1,\ldots,t.
$$
Then
$$
\norm{\nabla f(x_{t+1})}\le\frac{\lambda_{t+1}}2.
$$
\end{lemma}
In the following, suppose that
\cref{asm: convexity,asm: smoothness,asm:oracle} hold
and that the initialization and parameter conditions
of \cref{lem:gradient-below-clipping} are satisfied.

For $t=0,\ldots,K$, let $\mathcal E_t$ be the event
that, for every $i=1,\ldots,t$,
$$
\begin{cases}
f(y_i)\le f(x_i)
+\ip{\nabla f(x_i)}{y_i-x_i}
+4L_0\norm{y_i-x_i}^2,\\[2mm]
\displaystyle
\mathcal R_i(\bar x)\le\frac{cA_iL_0}{L_1^2},\\[2mm]
\displaystyle
\mathcal R_i(x^\star)\le\frac{R^2}{2}.
\end{cases}
$$
For $t=0$, there are no conditions, so
$\mathbf 1_{\mathcal E_0}=1$.

For $u\in\{\bar x,x^\star\}$, define
$$
\widehat{\mathcal R}_t(u) :=\sum_{i=1}^t\mathbf 1_{\mathcal E_{i-1}}
\left(\alpha_i\ip{\theta_i}{u-z_{i-1}}+\alpha_i^2\norm{\theta_i}^2\right).
$$
Since the events $\mathcal E_t$ are nested, on $\mathcal E_{t-1}$ we have
$$
\widehat{\mathcal R}_t(u)=\mathcal R_t(u).
$$

On $\mathcal E_{i-1}$,
\cref{lem:gradient-below-clipping} gives
$\norm{\nabla f(x_i)}\le\lambda_i/2$.
Since $\mathcal E_{i-1}\in\mathcal F_{i-1}$, applying
\citet[Lemma~5.1]{sadiev2023high} conditionally on
$\mathcal F_{i-1}$ yields, on $\mathcal E_{i-1}$,
$$
\norm{\mathbb E[\theta_i\mid\mathcal F_{i-1}]} \le\frac{2^\alpha\sigma^\alpha}{\lambda_i^{\alpha-1}}.
$$
Moreover, nonexpansiveness of clipping and
$\norm{g_i}\le\lambda_i$ imply, on $\mathcal E_{i-1}$,
$\norm{\theta_i} \le\min\left\{\norm{G_i-\nabla f(x_i)},\frac32\lambda_i
\right\}.$ Therefore, on $\mathcal E_{i-1}$,
$$
\mathbb E[\norm{\theta_i}^2\mid\mathcal F_{i-1}] \le \left(\frac32\lambda_i\right)^{2-\alpha} \mathbb E[\norm{G_i-\nabla f(x_i)}^\alpha \mid\mathcal F_{i-1}]\\
\le 2\sigma^\alpha\lambda_i^{2-\alpha}.
$$

To control $\widehat{\mathcal R}_t(u)$, we subtract the conditional mean from each summand. This separates the random fluctuations from the clipping bias and the conditional second-moment terms. For $u\in\{\bar x,x^\star\}$, we write
$$
\widehat{\mathcal R}_t(u)=M_t(u)+B_t(u),
$$
where 
$$
\begin{cases}
\displaystyle M_t(u):=\sum_{i=1}^t\mathbf 1_{\mathcal E_{i-1}}\Bigl(\alpha_i\ip{
\theta_i-\mathbb E[\theta_i\mid\mathcal F_{i-1}]}{u-z_{i-1}}+\alpha_i^2\bigl(\norm{\theta_i}^2
-\mathbb E[\norm{\theta_i}^2\mid\mathcal F_{i-1}]\bigr)\Bigr),\\[4mm]
\displaystyle B_t(u):=\sum_{i=1}^t\mathbf 1_{\mathcal E_{i-1}}\left(\alpha_i\ip{
\mathbb E[\theta_i\mid\mathcal F_{i-1}]}{u-z_{i-1}}+\alpha_i^2
\mathbb E[\norm{\theta_i}^2\mid\mathcal F_{i-1}]\right).
\end{cases}
$$
Since $\mathbf 1_{\mathcal E_{i-1}}$ and $u-z_{i-1}$ are $\mathcal F_{i-1}$-measurable, the increments of $M_t(u)$ have zero conditional mean. Thus, $(M_t(u))_{t=0}^K$ is a martingale, with $M_0(u)=0$.

The clipping bias and second-moment bounds give
$$
B_t(u)\le \sum_{i=1}^t\mathbf 1_{\mathcal E_{i-1}}\left(\frac{2^\alpha\sigma^\alpha\alpha_i}
{\lambda_i^{\alpha-1}}\norm{u-z_{i-1}}+2\sigma^\alpha\alpha_i^2\lambda_i^{2-\alpha}\right).
$$
For $i \ge 2$, on $\mathcal E_{i-1}$, applying \cref{lem:comparator}
at iteration $i-1$ with $u=\bar x$ and $u=x^\star$ gives
$$
\begin{cases}
\displaystyle \norm{\bar x-z_{i-1}} \le \frac{2\sqrt{cA_{i-1}L_0}}{L_1},\\[3mm]
\displaystyle \norm{x^\star-z_{i-1}}\le\sqrt2R.
\end{cases}
$$
For $i=1$, these bounds follow from $z_0=\bar x$,
$A_0=0$, and $\norm{\bar x-x^\star}\le R$. To control $M_t(u)$, we bound its increments and  the sum of their conditional variances, then apply  Freedman's inequality. A union bound over $t=1,\ldots,K$ and $u\in\{\bar x,x^\star\}$ gives the following result.

\begin{lemma}
\label{lem:stopped-residuals}
Suppose that Assumptions~\ref{asm: convexity},\ref{asm: smoothness} and~\ref{asm:oracle} hold with $L_0,L_1>0$. Assume that $\norm{\bar x-x^\star}\le R$ and $f(\bar x)-f^\star\le cL_0/L_1^2$, where $R>0$ and $0<c\le1/32$.
Assume moreover that $A_0=0$ and $y_0=z_0=\bar x$.
Let $K\ge1$, $\delta\in(0,1)$, and $\ell\ge 8\log(4K/\delta)$. Suppose that
$$
\begin{cases}
\displaystyle A_i\ge 96\ell^2L_0\alpha_i^2,
& i=1,\ldots,K,\\[2mm]
\displaystyle A_{i+1}\le\frac52A_i,
& i=1,\ldots,K-1.
\end{cases}
$$
For $i=1,\ldots,K$, define
$$
\begin{cases}
\displaystyle s_i:=\min\left\{ \frac{\sqrt{cA_iL_0}}{L_1},R\right\},\\[3mm]
\displaystyle \lambda_i:=\frac{s_i}{\alpha_i\ell}.
\end{cases}
$$
Suppose that $\sigma^\alpha\sum_{i=1}^K \frac{\alpha_i}{s_i\lambda_i^{\alpha-1}} \le\frac1{512}.$
Then, with probability at least $1-\delta$, simultaneously for every $t=1,\ldots,K$,
$$
\begin{cases}
\displaystyle \widehat{\mathcal R}_t(\bar x) \le\frac{cA_tL_0}{L_1^2},\\[3mm]
\displaystyle\widehat{\mathcal R}_t(x^\star) \le\frac{R^2}{2}.
\end{cases}
$$
On this event, $\mathcal E_K$ holds.
\end{lemma}

For the choice $\alpha_i=\frac{i+1}{aL_0},$ we have $A_i=\frac{i(i+3)}{2aL_0}
\ge\frac{aL_0}{2}\alpha_i^2.$ Hence, if $a\ge192\ell^2$, we have $A_i\ge96\ell^2L_0\alpha_i^2.$
Moreover, for all $i\ge1,$ we have $
\frac{A_{i+1}}{A_i}=\frac{(i+1)(i+4)}{i(i+3)}\le\frac52.$ Thus, the conditions on $A_i$ in \cref{lem:stopped-residuals} hold whenever
$a\ge192\ell^2$.

\subsection{Proofs}

\begin{proof}[Proof of \cref{lem:comparator}]
Fix $u\in\mathbb R^d$ and define
$$
\Phi_i(u):= A_i\left(f(y_i)-f(u)\right) +\frac12\norm{z_i-u}^2.
$$
By convexity, we have
$$
\begin{cases}
f(y_{i-1}) \ge f(x_i)+\ip{\nabla f(x_i)}{y_{i-1}-x_i},\\[0.5em]
f(u) \ge f(x_i)+\ip{\nabla f(x_i)}{u-x_i}.
\end{cases}
$$
Multiplying these inequalities by $A_{i-1}$ and
$\alpha_i$, respectively, and adding them yields
$$
A_{i-1}f(y_{i-1})+\alpha_i f(u) \ge A_i f(x_i) +\ip{\nabla f(x_i)}{A_{i-1}y_{i-1}+\alpha_i u-A_i x_i},
$$
where we used $A_i=A_{i-1}+\alpha_i$.
By the definition of $x_i$,
$A_i x_i=A_{i-1}y_{i-1}+\alpha_i z_{i-1}.$ Therefore
$$
A_i f(x_i)-A_{i-1}f(y_{i-1})-\alpha_i f(u) \le \alpha_i\ip{\nabla f(x_i)}{z_{i-1}-u}.
$$
We have $y_i-x_i=-\frac{\alpha_i^2}{A_i}g_i.$
Substituting this into the assumed descent inequality
and multiplying by $A_i$, we obtain
$$
A_i\bigl(f(y_i)-f(x_i)\bigr) \le -\alpha_i^2\langle\nabla f(x_i),g_i\rangle
+\frac{4L_0\alpha_i^4}{A_i}\|g_i\|^2.
$$
It follows that
$$A_i\bigl(f(y_i)-f(u)\bigr)-A_{i-1}\bigl(f(y_{i-1})-f(u)\bigr) \le
\alpha_i\langle\nabla f(x_i),z_{i-1}-u\rangle -\alpha_i^2\langle\nabla f(x_i),g_i\rangle
+\frac{4L_0\alpha_i^4}{A_i}\|g_i\|^2.$$
Since  $z_i=z_{i-1}-\alpha_i g_i,$ it holds that
$$
\frac12\|z_i-u\|^2-\frac12\|z_{i-1}-u\|^2
=-\alpha_i\langle g_i,z_{i-1}-u\rangle +\frac{\alpha_i^2}{2}\|g_i\|^2.
$$
Using $\theta_i=g_i-\nabla f(x_i)$, we obtain
$$
\Phi_i(u)-\Phi_{i-1}(u) \le \alpha_i\langle\theta_i,u-z_{i-1}\rangle+
\alpha_i^2\left[\langle\theta_i,g_i\rangle
-\left(\frac12-\frac{4L_0\alpha_i^2}{A_i}\right)\|g_i\|^2 \right],
$$
where we used  $-\langle\nabla f(x_i),g_i\rangle =-\langle g_i-\theta_i,g_i\rangle=\langle\theta_i,g_i\rangle-\|g_i\|^2.$

By assumption, $A_i\ge 16L_0\alpha_i^2,$ and therefore $\frac{4L_0\alpha_i^2}{A_i}\le\frac14.$  Consequently, 
$$
\langle\theta_i,g_i\rangle-\left(\frac12-\frac{4L_0\alpha_i^2}{A_i}\right)\|g_i\|^2\le
\langle\theta_i,g_i\rangle-\frac14\|g_i\|^2 =\|\theta_i\|^2-\frac14\|g_i-2\theta_i\|^2 \le \|\theta_i\|^2.
$$
We deduce that
$$
\Phi_i(u)-\Phi_{i-1}(u)\le \alpha_i\langle\theta_i,u-z_{i-1}\rangle +\alpha_i^2\|\theta_i\|^2.
$$
Summing over $i=1,\ldots,t$ yields
$$
\Phi_t(u)\le\Phi_0(u)+\mathcal R_t(u).
$$
Finally, $A_0=0$ and $z_0=\bar x$ imply
$\Phi_0(u)=\frac12\|\bar x-u\|^2$, so
$$
A_t\bigl(f(y_t)-f(u)\bigr)
+\frac12\|z_t-u\|^2
\le
\frac12\|\bar x-u\|^2+\mathcal R_t(u),
$$

\end{proof}

\begin{proof}[Proof of \cref{lem:grad_bound_smoothness}]
If $t=0$, then $x_1=\bar x$, and by
\cref{lem:self-bounding-norm-form}, we have
$$
\norm{\nabla f(x_1)}
\le\frac{c+\sqrt{c^2+2\nu c}}{\nu}\frac{L_0}{L_1}
\le(2\sqrt c+4c)\frac{L_0}{L_1}
\le\frac{L_0}{L_1},
$$
where we used $\nu\ge1/2$ and $c\le1/32$.

Now suppose $t\ge1$. Applying \cref{lem:comparator} with $u=\bar x$, we obtain, for every $j=1,\ldots,t$,
$$
A_j\bigl(f(y_j)-f^\star\bigr)+\frac12\norm{z_j-\bar x}^2
\le A_j\bigl(f(\bar x)-f^\star\bigr)+\mathcal R_j(\bar x)
\le\frac{2cA_jL_0}{L_1^2}.
$$
Then
$$
\begin{cases}
\displaystyle f(y_t)-f^\star\le\frac{2cL_0}{L_1^2},\\[2mm]
\displaystyle \norm{z_j-\bar x}
\le\sqrt{2A_j\bigl(f(\bar x)-f^\star\bigr)+2\mathcal R_j(\bar x)}
\le\sqrt{\frac{4cA_jL_0}{L_1^2}}=\frac{2\sqrt{cA_jL_0}}{L_1}.
\end{cases}
$$
By \cref{lem:self-bounding-norm-form} and $\nu\ge1/2$, we obtain
$$
\begin{aligned}
\norm{\nabla f(y_t)}
&\le\frac{2c+\sqrt{4c^2+4\nu c}}{\nu}\frac{L_0}{L_1}\\
&\le\left(4c+\sqrt{16c^2+8c}\right)\frac{L_0}{L_1}\\
&\le\frac{1+\sqrt{17}}8\frac{L_0}{L_1},
\end{aligned}
$$
where the last inequality uses $c\le\frac1{32}$.
Using  $y_t=A_t^{-1}\sum_{j=1}^t\alpha_jz_j$
and $A_j\le A_t$ for all $j \in \{1,\ldots,t \}$
$$
\norm{y_t-\bar x}
\le\frac1{A_t}\sum_{j=1}^t\alpha_j\norm{z_j-\bar x}
\le\frac{2\sqrt{cA_tL_0}}{L_1}.
$$
Therefore, using the definition of $x_{t+1}$ and the triangle inequality,
$$
\begin{aligned}
L_1\norm{x_{t+1}-y_t}
&=\frac{L_1\alpha_{t+1}}{A_{t+1}}\norm{z_t-y_t}\\
&\le\frac{4\alpha_{t+1}\sqrt{cA_tL_0}}{A_{t+1}}\\
&\le4\alpha_{t+1}\sqrt{\frac{cL_0}{A_{t+1}}}\\
&\le\sqrt c
\le\frac1{4\sqrt2},
\end{aligned}
$$
where we used $A_t\le A_{t+1}$,
$A_{t+1}\ge16L_0\alpha_{t+1}^2$,
and $c\le\frac1{32}$.

By \cref{lem:gradient-growth}, we obtain
$$
\begin{aligned}
L_0+L_1\norm{\nabla f(x_{t+1})}
&\le e^{L_1\norm{x_{t+1}-y_t}}
\bigl(L_0+L_1\norm{\nabla f(y_t)}\bigr)\\
&\le e^{1/(4\sqrt2)}\frac{9+\sqrt{17}}8L_0\\
&\le2L_0.
\end{aligned}
$$
Therefore, $\norm{\nabla f(x_{t+1})}\le\frac{L_0}{L_1}$.

Finally, for every $t \ge 0$, the update rule and clipping give
$$
L_1\norm{y_{t+1}-x_{t+1}}
=\frac{L_1\alpha_{t+1}^2}{A_{t+1}}\norm{g_{t+1}}
\le\frac{L_1\alpha_{t+1}^2}{A_{t+1}}\lambda_{t+1}
\le 1,
$$
where we used
$\lambda_{t+1}\le\frac{A_{t+1}}{L_1\alpha_{t+1}^2}$.
Applying \cref{lem:local-smoothness-short-segment}, we obtain
$$
\begin{aligned}
f(y_{t+1})
&\le f(x_{t+1})
+\ip{\nabla f(x_{t+1})}{y_{t+1}-x_{t+1}}+2\bigl(L_0+L_1\norm{\nabla f(x_{t+1})}\bigr)
\norm{y_{t+1}-x_{t+1}}^2\\
&\le f(x_{t+1}) +\ip{\nabla f(x_{t+1})}{y_{t+1}-x_{t+1}} +4L_0\norm{y_{t+1}-x_{t+1}}^2.
\end{aligned}
$$
\end{proof}

\begin{proof}[Proof of \cref{lem:gradient-below-clipping}]
By \cref{lem:self-bounding-norm-form} and $\nu\ge1/2$,
$$
f(x)-f^\star\le\frac{L_0}{16L_1^2}
\implies
\norm{\nabla f(x)}
\le3\sqrt{L_0\bigl(f(x)-f^\star\bigr)}.
$$

If $t=0$, then $x_1=\bar x$. By the preceding implication and the convexity bound
$f(\bar x)-f^\star\le R\norm{\nabla f(\bar x)}$, we obtain
$$
\norm{\nabla f(x_1)}
\le\min\left\{\frac{3\sqrt cL_0}{L_1},9L_0R\right\}.
$$
Since $A_1=A_0+\alpha_1=\alpha_1$ and
$A_1\ge96\ell^2L_0\alpha_1^2$, we have
$$
\alpha_1\ell\sqrt{\frac{L_0}{A_1}}\le\frac1{\sqrt{96}}
\quad\text{and}\quad
\alpha_1\ell L_0\le\frac1{96\ell}.
$$
Therefore,
$$
\alpha_1\ell\norm{\nabla f(x_1)} \le\min\left\{
\frac3{\sqrt{96}}\frac{\sqrt{cA_1L_0}}{L_1},\frac9{96\ell}R
\right\} \le\frac12\min\left\{\frac{\sqrt{cA_1L_0}}{L_1},R\right\},
$$
where we used $\ell\ge4$. Hence
$\norm{\nabla f(x_1)}\le\lambda_1/2$.

Now suppose $t\ge1$.
Since $A_i\ge96\ell^2L_0\alpha_i^2\ge16L_0\alpha_i^2$,
applying \cref{lem:comparator} with $u=\bar x$
and $u=x^\star$, we obtain
$$
f(y_t)-f^\star
\le\min\left\{\frac{2cL_0}{L_1^2},\frac{R^2}{A_t}\right\}.
$$
Similarly to the proof of \cref{lem:grad_bound_smoothness},
we obtain, for every $j=1,\ldots,t$,
$$
\begin{cases}
\displaystyle
\norm{z_j-\bar x}\le\frac{2\sqrt{cA_jL_0}}{L_1},\\[2mm]
\displaystyle
\norm{z_j-x^\star}\le\sqrt2R.
\end{cases}
$$
Since $y_t=A_t^{-1}\sum_{j=1}^t\alpha_jz_j$, it follows that
$$
\norm{z_t-y_t}
\le\min\left\{\frac{4\sqrt{cA_tL_0}}{L_1},2\sqrt2R\right\}.
$$

As in the proof of \cref{lem:grad_bound_smoothness}, we obtain $\norm{\nabla f(y_t)}\le\frac{3L_0}{4L_1}.$ Moreover,
$$
\begin{aligned}
L_1\norm{x_{t+1}-y_t}
&=\frac{L_1\alpha_{t+1}}{A_{t+1}}\norm{z_t-y_t}\\
&\le\frac{4\alpha_{t+1}\sqrt{cA_tL_0}}{A_{t+1}}\\
&\le4\sqrt c\,\alpha_{t+1}\sqrt{\frac{L_0}{A_{t+1}}}\\
&\le\frac{4\sqrt c}{\sqrt{96}\ell}
\le\frac18,
\end{aligned}
$$
where we used $A_t\le A_{t+1}$,
$A_{t+1}\ge96\ell^2L_0\alpha_{t+1}^2$,
$c\le1/32$, and $\ell\ge4$.

By \cref{lem:gradient-growth}, every $v\in[y_t,x_{t+1}]$
therefore satisfies
$L_0+L_1\norm{\nabla f(v)}\le\frac74e^{1/8}L_0\le2L_0$.
Hence, by \cref{asm: smoothness},
$$
\begin{aligned}
\norm{\nabla f(x_{t+1})}&\le\norm{\nabla f(y_t)}+\norm{\nabla f(x_{t+1})-\nabla f(y_t)}\\
&\le\norm{\nabla f(y_t)}+\left(L_0+L_1\sup_{v\in[y_t,x_{t+1}]}\norm{\nabla f(v)}\right)
\norm{x_{t+1}-y_t}\\
&\le\norm{\nabla f(y_t)}+2L_0\norm{x_{t+1}-y_t}\\
&\le3\sqrt{L_0\bigl(f(y_t)-f^\star\bigr)}+\frac{2L_0\alpha_{t+1}}{A_{t+1}}\norm{z_t-y_t}.
\end{aligned}
$$
Substituting the preceding bounds, we obtain
$$
\begin{cases}
\displaystyle \norm{\nabla f(x_{t+1})} \le
\left(3\sqrt2+\frac8{\sqrt{96}\ell}\right) \frac{\sqrt cL_0}{L_1},\\[3mm]
\displaystyle \norm{\nabla f(x_{t+1})} \le
\left(3\sqrt{\frac52}+\frac{4\sqrt2}{\sqrt{96}\ell}\right)
R\sqrt{\frac{L_0}{A_{t+1}}},
\end{cases}
$$
where we used $A_t\le A_{t+1}\le\frac52A_t$
and $A_{t+1}\ge96\ell^2L_0\alpha_{t+1}^2$.

We have $\lambda_{t+1}:=\frac1{\alpha_{t+1}\ell}\min\left\{\frac{\sqrt{cA_{t+1}L_0}}{L_1},R
\right\}.$ Moreover, $A_{t+1}\ge96\ell^2L_0\alpha_{t+1}^2$
implies
$$
\alpha_{t+1}\ell\sqrt{\frac{L_0}{A_{t+1}}} \le\frac1{\sqrt{96}}.
$$
Multiplying the two gradient bounds above by
$\alpha_{t+1}\ell$, we  obtain
$$
\begin{cases}
\displaystyle \alpha_{t+1}\ell\norm{\nabla f(x_{t+1})}
\le\left(\frac{\sqrt3}{4}+\frac1{12\ell}\right)
\frac{\sqrt{cA_{t+1}L_0}}{L_1}
\le\left(\frac{\sqrt3}{4}+\frac1{48}\right)
\frac{\sqrt{cA_{t+1}L_0}}{L_1}
<\frac12\frac{\sqrt{cA_{t+1}L_0}}{L_1},\\[3mm]
\displaystyle \alpha_{t+1}\ell\norm{\nabla f(x_{t+1})}
\le\left(\frac{\sqrt{15}}8+\frac{\sqrt2}{24\ell}\right)R
\le\left(\frac{\sqrt{15}}8+\frac{\sqrt2}{96}\right)R
<\frac12R.
\end{cases}
$$
Therefore
$$
\norm{\nabla f(x_{t+1})} \le \frac1{2\alpha_{t+1}\ell}
\min\left\{\frac{\sqrt{cA_{t+1}L_0}}{L_1},R\right\}=\frac{\lambda_{t+1}}2.
$$
\end{proof}

\begin{proof}[Proof of \cref{lem:stopped-residuals}]
For $t=1,\ldots,K$, let $S_t:=\sigma^\alpha\sum_{i=1}^t\frac{\alpha_i}{s_i\lambda_i^{\alpha-1}}.$
By assumption, $S_t\le S_K\le 1/512$.

For $i\ge2$, on $\mathcal E_{i-1}$, Lemma~\ref{lem:comparator} gives
$$
\begin{cases}
\norm{\bar x-z_{i-1}}\le 2 \frac{\sqrt{cA_{i-1}L_0}}{L_1} \le   2\frac{\sqrt{cA_iL_0}}{L_1},\\[2mm]
\norm{x^\star-z_{i-1}}\le \sqrt{2}R.
\end{cases}
$$
We also have $
\norm{\bar x-z_{i-1}} \le \norm{\bar x-x^\star}+\norm{x^\star-z_{i-1}} \le (1+\sqrt2)R.$

By the definition of $s_i$, for every $i$, we have
$$
\begin{cases}
\norm{\bar x-z_{i-1}}\le (1+\sqrt{2})s_i\le \frac52 s_i,\\[2mm]
\norm{x^\star-z_{i-1}}\le \sqrt{2}R.
\end{cases}
$$
The same bounds hold for $i=1$ since $z_0=\bar x$.

We have
$$
B_t(u)\le \sum_{i=1}^t\1_{\mathcal E_{i-1}} \left(\frac{2^\alpha\sigma^\alpha\alpha_i}
{\lambda_i^{\alpha-1}}\norm{u-z_{i-1}}+2\sigma^\alpha\alpha_i^2\lambda_i^{2-\alpha}
\right).
$$

Using $\alpha_i\lambda_i=s_i/\ell$, we have $\sigma^\alpha\alpha_i^2\lambda_i^{2-\alpha}
=\frac{\sigma^\alpha\alpha_i}{\lambda_i^{\alpha-1}}\frac{s_i}{\ell}.$ Hence, since $2^\alpha\le4$,
$$
B_t(\bar x)
\le \left(10+\frac{2}{\ell}\right)\sum_{i=1}^t\frac{\sigma^\alpha\alpha_i}{s_i\lambda_i^{\alpha-1}}s_i^2.
$$
Since $s_i^2\le cA_iL_0/L_1^2\le cA_tL_0/L_1^2$,
$$
B_t(\bar x) \le\left(10+\frac{2}{\ell}\right) \frac{cA_tL_0}{L_1^2}S_t <\frac1{32}\frac{cA_tL_0}{L_1^2},
$$
where we used $\ell\ge4$ and $S_t\le1/512$.

Similarly,
$$
B_t(x^\star) \le \left(4\sqrt2+\frac{2}{\ell}\right) R^2S_t <\frac{R^2}{64}.
$$
Recall that $\theta_i:=g_i-\nabla f(x_i),$ where $G_i:=\nabla f(x_i,\xi_i)$ and $g_i:=\operatorname{clip}(G_i,\lambda_i)$.
We now control the martingale term $M_t(u)$. Its $i$-th summand is
$$
\1_{\mathcal E_{i-1}}\left(\alpha_i\left\langle\theta_i-\E[\theta_i\mid\mathcal F_{i-1}],
u-z_{i-1}\right\rangle+\alpha_i^2\left(\norm{\theta_i}^2-\E[\norm{\theta_i}^2\mid\mathcal F_{i-1}]
\right)\right).
$$
On $\mathcal E_{i-1}$,
we have $\norm{\theta_i}\le \frac32\lambda_i$, and therefore
$$
\norm{\theta_i-\E[\theta_i\mid\mathcal F_{i-1}]} \le 3\lambda_i,
$$
and
$$
\left|\norm{\theta_i}^2-\E[\norm{\theta_i}^2\mid\mathcal F_{i-1}]\right|\le \frac94\lambda_i^2.
$$
Hence, for $u=\bar x$, the absolute value of the $i$-th summand
of $M_t(\bar x)$ is bounded by
$$
3\alpha_i\lambda_i\norm{\bar x-z_{i-1}}+\frac94\alpha_i^2\lambda_i^2\le\frac{15}{2}\alpha_i\lambda_i s_i+\frac94\alpha_i^2\lambda_i^2.
$$
Using $\alpha_i\lambda_i=s_i/\ell$, for every $i\le t$,
$$
\frac{15}{2}\alpha_i\lambda_i s_i+\frac94\alpha_i^2\lambda_i^2=\left(\frac{15}{2\ell}+\frac{9}{4\ell^2}\right)s_i^2
\le \frac9\ell\frac{cA_tL_0}{L_1^2}.
$$
Similarly, for $u=x^\star$, the absolute value of the $i$-th
summand is bounded by
$$
3\sqrt2\,\alpha_i\lambda_iR +\frac94\alpha_i^2\lambda_i^2=\frac{3\sqrt2}{\ell}s_iR
+\frac{9}{4\ell^2}s_i^2 \le \frac{5R^2}{\ell}.
$$
We next bound the conditional variances. Since 
$$
\begin{cases}
\displaystyle \E\!\left[\norm{\theta_i-\E[\theta_i\mid\mathcal F_{i-1}]}^2
\mid\mathcal F_{i-1}\right] \le \E[\norm{\theta_i}^2\mid\mathcal F_{i-1}]
\le 2\sigma^\alpha\lambda_i^{2-\alpha},\\[4mm]
\displaystyle \E\!\left[ \left(\norm{\theta_i}^2-\E[\norm{\theta_i}^2\mid\mathcal F_{i-1}]
\right)^2 \Bigm|\mathcal F_{i-1} \right]\le \E[\norm{\theta_i}^4\mid\mathcal F_{i-1}]
\le \frac92\sigma^\alpha\lambda_i^{4-\alpha},
\end{cases}
$$
using $(a+b)^2\le2a^2+2b^2$, the conditional variance of the $i$-th summand of $M_t(u)$ is at most
$$
4\sigma^\alpha\alpha_i^2\lambda_i^{2-\alpha} \norm{u-z_{i-1}}^2+ 9\sigma^\alpha\alpha_i^4
\lambda_i^{4-\alpha}.
$$

Using again $\alpha_i\lambda_i=s_i/\ell$, we have
$$
\sigma^\alpha\alpha_i^2\lambda_i^{2-\alpha}=\frac{\sigma^\alpha\alpha_i}
{s_i\lambda_i^{\alpha-1}}\frac{s_i^2}{\ell},
$$
and
$$
\sigma^\alpha\alpha_i^4\lambda_i^{4-\alpha}=\frac{\sigma^\alpha\alpha_i}{s_i\lambda_i^{\alpha-1}}\frac{s_i^4}{\ell^3}.
$$
Therefore,
$$
\sum_{i=1}^t \E\left[\text{($i$-th summand of $M_t(\bar x)$)}^2
\mid\mathcal F_{i-1}\right] \le \left(\frac{25}{\ell}+\frac9{\ell^3}\right)
\left(\frac{cA_tL_0}{L_1^2}\right)^2S_t \le \frac{26}{512\ell} \left(\frac{cA_tL_0}{L_1^2}\right)^2,
$$
where we used $\ell\ge4$ and $S_t\le1/512$.

Similarly, we obtain
$$
\sum_{i=1}^t\E\left[ \text{($i$-th summand of $M_t(x^\star)$)}^2\mid\mathcal F_{i-1}
\right] \le \left(\frac8\ell+\frac9{\ell^3}\right)R^4S_t \le \frac{9R^4}{512\ell}.
$$

Fix $t\in\{1,\ldots,K\}$. For $u=\bar x$, the martingale differences are bounded in absolute value by $\frac9\ell\frac{cA_tL_0}{L_1^2}$, while their conditional variances sum to at most $\frac{26}{512\ell}\left(\frac{cA_tL_0}{L_1^2}\right)^2$. Hence, by \cref{cor:freedman-hp}, with probability at least $1-\delta/(2K)$,
$$
|M_t(\bar x)|\le \frac{cA_tL_0}{L_1^2}\left(\sqrt{\frac{52}{512\ell}\log\frac{4K}{\delta}}+\frac6\ell\log\frac{4K}{\delta}\right).
$$
Since $\ell\ge8\log(4K/\delta)$, the term in parentheses is at most $\frac{\sqrt{13}}{32}+\frac34<\frac78$. Thus
$$
|M_t(\bar x)|<\frac78\frac{cA_tL_0}{L_1^2}.
$$
Combining this with $B_t(\bar x)<\frac1{32}\frac{cA_tL_0}{L_1^2}$, we obtain
$$
\widehat{\mathcal R}_t(\bar x)=M_t(\bar x)+B_t(\bar x)<\left(\frac78+\frac1{32}\right)\frac{cA_tL_0}{L_1^2} <\frac{cA_tL_0}{L_1^2}.
$$

Similarly, for $u=x^\star$, the martingale differences are bounded in absolute value by $5R^2/\ell$, while their conditional variances sum to at most $9R^4/(512\ell)$. Therefore, with probability at least $1-\delta/(2K)$,
$$
|M_t(x^\star)|\le R^2\left(\sqrt{\frac{18}{512\ell}\log\frac{4K}{\delta}}+\frac{10}{3\ell}\log\frac{4K}{\delta}\right).
$$
Since $\ell\ge8\log(4K/\delta)$, the term in parentheses is at most $\frac{3}{16\sqrt8}+\frac5{12}<\frac{29}{60}$. Hence,
$$
\widehat{\mathcal R}_t(x^\star)\le |M_t(x^\star)|+B_t(x^\star) <\left(\frac{29}{60}+\frac1{64}\right)R^2 <\frac{R^2}{2}.
$$

Taking a union bound over $t=1,\ldots,K$ and $u\in\{\bar x,x^\star\}$, with probability at least $1-\delta$, simultaneously for every $t=1,\ldots,K$,
$$
\begin{cases}
\widehat{\mathcal R}_t(\bar x)\le \dfrac{cA_tL_0}{L_1^2},\\[2mm]
\widehat{\mathcal R}_t(x^\star)\le \dfrac{R^2}{2}.
\end{cases}
$$

It remains to show that $\mathcal E_K$ holds on this event. We argue by induction. For $t=0$, there are no conditions, so $\1_{\mathcal E_0}=1$. Suppose that $\mathcal E_{t-1}$ holds for some $t\in\{1,\ldots,K\}$. Since
$A_t\ge96\ell^2L_0\alpha_t^2$, we have $A_t\ge16L_0\alpha_t^2$. Moreover, by the definition of $\lambda_t$,
$$
\lambda_t
\le\frac{\sqrt{cA_tL_0}}{\alpha_t\ell L_1}
\le\frac{A_t}{L_1\alpha_t^2},
$$
where the last inequality follows from
$A_t\ge96\ell^2L_0\alpha_t^2$.
Hence, applying \cref{lem:grad_bound_smoothness} with $t-1$, we obtain
$$
f(y_t)\le f(x_t)+\langle\nabla f(x_t),y_t-x_t\rangle+4L_0\norm{y_t-x_t}^2.
$$
Since $\mathcal E_{t-1}$ holds, $\widehat{\mathcal R}_t(u)=\mathcal R_t(u)$ for $u\in\{\bar x,x^\star\}$. Therefore,
$$
\begin{cases}
\displaystyle \mathcal R_t(\bar x)\le\frac{cA_tL_0}{L_1^2},\\[2mm]
\displaystyle \mathcal R_t(x^\star)\le\frac{R^2}{2}.
\end{cases}
$$
Thus all the conditions defining $\mathcal E_t$ hold. By induction, $\mathcal E_K$ holds.

\end{proof}

\begin{proof}[Proof of \cref{thm:main}]
Let $s_i:=\min\{\sqrt{cA_iL_0}/L_1,R\}$. We have 
$\lambda_i=s_i/(\alpha_i\ell)$, then
$$
\sigma^\alpha\sum_{i=1}^K\frac{\alpha_i}{s_i\lambda_i^{\alpha-1}}
=
\sigma^\alpha\ell^{\alpha-1}\sum_{i=1}^K
\left(\frac{\alpha_i}{s_i}\right)^\alpha.
$$
Since $s_i=\min\{\sqrt{cA_iL_0}/L_1,R\}$,
$\left(\alpha_i/s_i\right)^\alpha
\le\left(\alpha_iL_1/\sqrt{cA_iL_0}\right)^\alpha
+(\alpha_i/R)^\alpha$.
Using $A_i\ge(aL_0/2)\alpha_i^2$, $\alpha_i=(i+1)/(aL_0)$, and
$\sum_{i=1}^K(i+1)^\alpha\le2^\alpha K^{\alpha+1}$, we obtain
$$
\sigma^\alpha\sum_{i=1}^K\frac{\alpha_i}{s_i\lambda_i^{\alpha-1}}\le
\sigma^\alpha\ell^{\alpha-1} \left[K\left(\frac{2}{ca}\right)^{\alpha/2}
\left(\frac{L_1}{L_0}\right)^\alpha
+\frac{2^\alpha K^{\alpha+1}}{a^\alpha L_0^\alpha R^\alpha}\right].
$$
By the choice of $a$, together with
$q^\alpha=1024\ell^{\alpha-1}K$, each of the two terms above is bounded by $1/1024$. Therefore their sum is at most $1/512$. Since $a\ge192\ell^2$ as well, all the conditions of \cref{lem:stopped-residuals} are satisfied. Hence, with probability at least $1-\delta$, the event $\mathcal E_K$ holds.

On $\mathcal E_K$, applying \cref{lem:comparator} with $u=x^\star$ gives
$f(y_K)-f^\star\le R^2/A_K=2aL_0R^2/[K(K+3)]$. Applying the same lemma with $u=\bar x$ yields
$f(y_K)-f^\star\le2cL_0/L_1^2$. Combining the two bounds, we obtain
$$
f(y_K)-f^\star \le \min\left\{\frac{2cL_0}{L_1^2},\frac{2aL_0R^2}{K(K+3)}\right\}.
$$
Using $K(K+3)\ge K^2$, $\max\{x,y,z\}\le x+y+z$, and the definition of
$a$, we get
$$
\frac{2aL_0R^2}{K(K+3)} \le \frac{384L_0R^2\ell^2}{K^2}+\frac{4q^2\sigma^2L_1^2R^2}{cL_0K^2}
+\frac{4q\sigma R}{K}.
$$
Therefore, using
$\min\{x,r+s+t\}\le r+t+\min\{x,s\}$ for nonnegative quantities,
$$
f(y_K)-f^\star \le\frac{384L_0R^2\ell^2}{K^2}+\frac{4q\sigma R}{K}+
\min\left\{\frac{2cL_0}{L_1^2},\frac{4q^2\sigma^2L_1^2R^2}{cL_0K^2}\right\}.
$$
Finally, using $\min\{x,y\}\le\sqrt{xy}$,
$$
\min\left\{ \frac{2cL_0}{L_1^2}, \frac{4q^2\sigma^2L_1^2R^2}{cL_0K^2}\right\}
\le \frac{2\sqrt2\,q\sigma R}{K}.
$$
Hence,
$$
f(y_K)-f^\star \le \frac{384L_0R^2\ell^2}{K^2} +(4+2\sqrt2)\frac{q\sigma R}{K}.
$$
\end{proof}

\begin{proof}[Proof of \cref{cor:end-to-end}]
Let $M:=\max\left\{L_1R_0,\log\frac{24}{\delta}\right\}.$ For Phase I, we run \algname{Clip-SGD} with the parameters of \cite[Theorem~2]{kadi2026highprobabilityconvergenceclippedsgd} for $K_1$ iterations, where $K_1$ is chosen large enough so that
$$
K_1\ge (L_1R_0)^2, \qquad K_1\ge L_1R_0 M^{\frac{\alpha-1}{2\alpha}},
$$
and the iteration requirement of the cited result is satisfied. 

Since $\alpha\in(1,2]$, we have $0<\frac{\alpha-1}{2\alpha}\le \frac14.$ Then
$$
L_1R_0 M^{\frac{\alpha-1}{2\alpha}} \le \max\left\{ (L_1R_0)^{1+\frac{\alpha-1}{2\alpha}},
\,L_1R_0\left(\log\frac{24}{\delta}\right)^{\frac{\alpha-1}{2\alpha}}\right\}
= \widetilde{\mathcal O}\left(1+(L_1R_0)^2\right).
$$
The additional iteration requirement in
\cite[Theorem~2]{kadi2026highprobabilityconvergenceclippedsgd}
is also covered by the same bound up to logarithmic factors. Therefore,
it is enough to take
$$
K_1=
\widetilde{\mathcal O}\left(1+(L_1R_0)^2\right).
$$

By \cite{kadi2026highprobabilityconvergenceclippedsgd},  with probability at least $1-\delta/2$, the output $\bar x$ of Phase I satisfies
$\|\bar x-x^\star\|\le 2R_0$ and $f(\bar x)-f^\star\le B,$
where, for some universal constant $C>0$,
$$
B:=C\max\left\{\frac{L_0R_0^2}{K_1},\,\sigma R_0\left(\frac{M}{K_1}\right)^{\frac{\alpha-1}{\alpha}}
\right\}.
$$

Let $\bar L_0:=L_0+32L_1^2B.$ Since $\bar L_0\ge L_0$, Assumption~\ref{asm: smoothness} also holds
with $(\bar L_0,L_1)$ in place of $(L_0,L_1)$. Moreover,
$$
f(\bar x)-f^\star \le B \le \frac{\bar L_0}{32L_1^2}.
$$
Hence Theorem~\ref{thm:main} can be applied in Phase II with $L_0$
replaced by $\bar L_0$, $R=2R_0$, and failure probability $\delta/2$.
Therefore, Phase II reaches accuracy $\varepsilon$ after
$$
\widetilde{\mathcal O}\left(\sqrt{\frac{\bar L_0R_0^2}{\varepsilon}}+\left(\frac{\sigma R_0}{\varepsilon}\right)^{\frac{\alpha}{\alpha-1}}\right)
$$
stochastic-gradient evaluations.

Now we bound the first term. Using $\sqrt{a+b}\le\sqrt a+\sqrt b$, we obtain

$$
\sqrt{\frac{\bar L_0R_0^2}{\varepsilon}} \lesssim \sqrt{\frac{L_0R_0^2}{\varepsilon}}
+ L_1R_0\sqrt{\frac{B}{\varepsilon}}.
$$
Using the definition of $B$, we have
$$
L_1R_0\sqrt{\frac{B}{\varepsilon}} \lesssim \frac{L_1R_0}{\sqrt{K_1}}\sqrt{\frac{L_0R_0^2}{\varepsilon}}+L_1R_0  \sqrt{\frac{\sigma R_0}{\varepsilon}} \left(\frac{M}{K_1}\right)^{\frac{\alpha-1}{2\alpha}}.
$$
since $K_1\gtrsim (L_1R_0)^2$, the first term is bounded by $\sqrt{L_0R_0^2/\varepsilon}.$ For the second term, we have
$$
L_1R_0 \sqrt{\frac{\sigma R_0}{\varepsilon}} \left(\frac{M}{K_1}\right)^{\frac{\alpha-1}{2\alpha}}
=L_1R_0M^{\frac{\alpha-1}{2\alpha}}\left(\frac{\left(\frac{\sigma R_0}{\varepsilon}\right)^{\frac{\alpha}{\alpha-1}}}{K_1}\right)^{\frac{\alpha-1}{2\alpha}}.
$$
Since $L_1R_0M^{\frac{\alpha-1}{2\alpha}}=\widetilde{\mathcal O}(K_1)$ and $0<(\alpha-1)/(2\alpha)\le1,$
the inequality $x^r\le1+x$ for $x\ge0$ and $0<r\le1$ gives
$$
L_1R_0\sqrt{\frac{\sigma R_0}{\varepsilon}} \left(\frac{M}{K_1}\right)^{\frac{\alpha-1}{2\alpha}}
=\widetilde{\mathcal O}\left(K_1+\left(\frac{\sigma R_0}{\varepsilon}\right)^{\frac{\alpha}{\alpha-1}}
\right).
$$
Consequently,
$$
\sqrt{\frac{\bar L_0R_0^2}{\varepsilon}}=\widetilde{\mathcal O}\left(
\sqrt{\frac{L_0R_0^2}{\varepsilon}}
+K_1+\left(\frac{\sigma R_0}{\varepsilon}\right)^{\frac{\alpha}{\alpha-1}}\right).
$$
Combining the two phases and using
$$
K_1=\widetilde{\mathcal O}\left(1+(L_1R_0)^2\right)
$$
we obtain a total complexity of
$$
\widetilde{\mathcal O}\left(1+(L_1R_0)^2 +\sqrt{\frac{L_0R_0^2}{\varepsilon}} +\left(\frac{\sigma R_0}{\varepsilon}\right)^{\frac{\alpha}{\alpha-1}}\right).
$$
\end{proof}

\section{High-Probability Analysis for the One-Phase Result}
\label{app:one-phase}

Throughout this section, assume
$$
\begin{cases}
\norm{x_0-x^\star}\le R_0,\\[0.5em]
\Delta_0:=f(x_0)-f^\star\le \bar\Delta_0.
\end{cases}
$$
We also use
$$
\mathcal R_t(u):=\sum_{i=1}^t\alpha_i\langle\theta_i,u-z_{i-1}\rangle+
\sum_{i=1}^t \alpha_i^2\norm{\theta_i}^2, \qquad
u\in\mathbb R^d,$$
as introduced in the two-phase analysis.

\begin{lemma}
\label{lem:comparator2}
Suppose that \cref{asm: convexity} holds. Let $A_0\ge0$, $y_0=z_0=v_0$, and let $t\in\{1,\ldots,K\}$.
Assume that, for every $i=1,\ldots,t$,
$$
f(y_i)\le f(x_i) +\langle\nabla f(x_i),y_i-x_i\rangle+\frac{A_i}{4\alpha_i^2}\norm{y_i-x_i}^2.
$$
Then, for every $u\in\mathbb R^d$,
$$
A_t\bigl(f(y_t)-f(u)\bigr)+\frac12\norm{z_t-u}^2\le A_0\bigl(f(v_0)-f(u)\bigr)
+\frac12\norm{v_0-u}^2+\mathcal R_t(u).
$$
\end{lemma}

\begin{lemma}
\label{lem:one-phase-descent}
Suppose that Assumptions \ref{asm: convexity} and \ref{asm: smoothness} hold,
where $L_0,L_1>0$. Let $v_0=x_0$, $\ell\ge1$,
$a\ge32$, and suppose that $A_0\Delta_0\le\frac{R_0^2}{64}.$
Assume that, for every $k=0,\ldots,K-1$,
$$
\begin{cases}
\dfrac{\alpha_{k+1}}{A_k}\le  \dfrac{1}{64\ell(1+L_1R_0)},\\[0.8em]
A_k\ge aL_0\alpha_{k+1}^2.
\end{cases}
$$
For $i=1,\ldots,K$, let
$$
\lambda_i:=\frac{R_0}{\alpha_i\ell}.
$$
Let $t\in\{0,\ldots,K-1\}$ and suppose that, for every
$i=1,\ldots,t$,
$$
\begin{cases}
f(y_i) \le f(x_i)+\langle\nabla f(x_i),y_i-x_i\rangle +\dfrac{A_i}{4\alpha_i^2}\norm{y_i-x_i}^2,\\[0.8em]
\mathcal R_i(x^\star)\le\dfrac{R_0^2}{2}.
\end{cases}
$$
Then
$$
f(y_{t+1}) \le f(x_{t+1})+\langle\nabla f(x_{t+1}),y_{t+1}-x_{t+1}\rangle
+\frac{A_{t+1}}{4\alpha_{t+1}^2}\norm{y_{t+1}-x_{t+1}}^2.
$$
\end{lemma}

\begin{lemma}
\label{lem:one-phase-gradient}
Suppose that Assumptions \ref{asm: convexity} and \ref{asm: smoothness} hold,
where $L_0,L_1>0$. Consider the iterates of Algorithm~1 with
$v_0=x_0$, $\ell\ge1$, $a\ge128\ell^2$, and suppose that $A_0\Delta_0\le\frac{R_0^2}{64}.$
Assume that, for every $k=0,\ldots,K-1$,
$$
\begin{cases}
\dfrac{\alpha_{k+1}}{A_k} \le \dfrac{1}{64\ell(1+L_1R_0)},\\[0.8em]
A_k\ge aL_0\alpha_{k+1}^2.
\end{cases}
$$
For $i=1,\ldots,K$, let
$$
\lambda_i:=\frac{R_0}{\alpha_i\ell}.
$$
Let $t\in\{0,\ldots,K-1\}$ and suppose that, for every
$i=1,\ldots,t$,
$$
\begin{cases}
f(y_i) \le f(x_i)+\langle\nabla f(x_i),y_i-x_i\rangle +\dfrac{A_i}{4\alpha_i^2}\norm{y_i-x_i}^2,\\[0.8em]
\mathcal R_i(x^\star)\le\dfrac{R_0^2}{2}.
\end{cases}
$$
Then
$$
\norm{\nabla f(x_{t+1})} \le \frac{\lambda_{t+1}}{2}.
$$
\end{lemma}
In the following, suppose that Assumptions \ref{asm: convexity}, \ref{asm: smoothness}, and \ref{asm:oracle} hold, and that the initialization and parameter conditions of \cref{lem:one-phase-gradient} are satisfied. Let $\mathcal F_0$ be the trivial $\sigma$-algebra and,
for $i=1,\ldots,K$, define
$$
\mathcal F_i:=\sigma(\xi_1,\ldots,\xi_i).
$$
We use the stochastic notation $G_i,g_i,\theta_i$ introduced in the two-phase analysis.

For $t=0,\ldots,K$, let $\mathcal E_t$ be the event that, for every $i=1,\ldots,t$,
$$
\begin{cases}
f(y_i) \le f(x_i)+\langle\nabla f(x_i),y_i-x_i\rangle +\dfrac{A_i}{4\alpha_i^2}\norm{y_i-x_i}^2,\\[0.8em]
\mathcal R_i(x^\star)\le\dfrac{R_0^2}{2}.
\end{cases}
$$
For $t=0$, there are no conditions, so $\mathbf 1_{\mathcal E_0}=1$. Define
$$
\widehat{\mathcal R}_t(x^\star) := \sum_{i=1}^t \mathbf 1_{\mathcal E_{i-1}} \left(
\alpha_i\langle\theta_i,x^\star-z_{i-1}\rangle + \alpha_i^2\norm{\theta_i}^2 \right).
$$
Since the events $\mathcal E_t$ are nested, on $\mathcal E_{t-1}$ we have $
\widehat{\mathcal R}_t(x^\star)=\mathcal R_t(x^\star).$ Moreover, on $\mathcal E_{i-1}$, \cref{lem:one-phase-gradient} gives
$$
\norm{\nabla f(x_i)}
\le
\frac{\lambda_i}{2}.
$$
Therefore, on $\mathcal E_{i-1}$, the same clipping estimates as in the
two-phase analysis apply:
$$
\begin{cases}
\norm{\mathbb E[\theta_i\mid\mathcal F_{i-1}]} \le \dfrac{2^\alpha\sigma^\alpha}{\lambda_i^{\alpha-1}},\\[0.8em]
\mathbb E[\norm{\theta_i}^2\mid\mathcal F_{i-1}] \le 2\sigma^\alpha\lambda_i^{2-\alpha},\\[0.8em]
\norm{\theta_i} \le \dfrac32\lambda_i.
\end{cases}
$$

To control $\widehat{\mathcal R}_t(x^\star)$, we separate the martingale
fluctuations from the clipping bias and the conditional second-moment terms.
We write
$$
\widehat{\mathcal R}_t(x^\star)=M_t+B_t,
$$
where
$$
\begin{cases}
M_t:= \displaystyle \sum_{i=1}^t \mathbf 1_{\mathcal E_{i-1}}\left(\alpha_i\left\langle
\theta_i-\mathbb E[\theta_i\mid\mathcal F_{i-1}],x^\star-z_{i-1}\right\rangle
+\alpha_i^2\left(\norm{\theta_i}^2-\mathbb E[\norm{\theta_i}^2\mid\mathcal F_{i-1}]\right)\right),\\[1.2em]
B_t:=\displaystyle\sum_{i=1}^t\mathbf 1_{\mathcal E_{i-1}}\left(\alpha_i\left\langle
\mathbb E[\theta_i\mid\mathcal F_{i-1}],x^\star-z_{i-1}\right\rangle+\alpha_i^2
\mathbb E[\norm{\theta_i}^2\mid\mathcal F_{i-1}]\right).
\end{cases}
$$
Since $\mathbf 1_{\mathcal E_{i-1}}$ and $x^\star-z_{i-1}$ are
$\mathcal F_{i-1}$-measurable, the increments of $M_t$ have zero
conditional mean. Hence $(M_t)_{t=0}^K$ is a martingale with $M_0=0$.

\begin{lemma}
\label{lem:one-phase-residual}
Suppose that Assumptions \ref{asm: convexity},
\ref{asm: smoothness}, and \ref{asm:oracle} hold, where $L_0,L_1>0$. Assume that $\norm{x_0-x^\star}\le R_0$ and $A_0\Delta_0\le\frac{R_0^2}{64}.$ Let $K\ge1$, $\delta\in(0,1), \ell\ge 8\log\frac{4K}{\delta},a\ge128\ell^2,$ and suppose that, for every $k=0,\ldots,K-1$,
$$
\begin{cases}
\dfrac{\alpha_{k+1}}{A_k} \le \dfrac{1}{64\ell(1+L_1R_0)},\\[0.8em]
A_k\ge aL_0\alpha_{k+1}^2.
\end{cases}
$$
For $i=1,\ldots,K$, let $
\lambda_i:=\frac{R_0}{\alpha_i\ell}.$ Suppose moreover that $$
\frac{\sigma^\alpha\ell^{\alpha-1}}{R_0^\alpha} \sum_{i=1}^K\alpha_i^\alpha \le \frac{1}{512}.$$

Then, with probability at least $1-\delta$, simultaneously for every
$t=1,\ldots,K$,
$$
\widehat{\mathcal R}_t(x^\star) \le \frac{R_0^2}{2}.
$$
On this event, $\mathcal E_K$ holds.
\end{lemma}

\begin{lemma}
\label{lem:one-phase-weights}
Let $L_0>0$, $L_1\ge0$, $R_0>0$, $\ell\ge1$, $a>0$, and
$\alpha\in(1,2]$. Suppose that $0<A_0\le 1/(aL_0)$ and, for every $k\ge0$,
$$
\begin{cases}
\alpha_{k+1}=\min\left\{\dfrac{A_k}{64\ell(1+L_1R_0)}, \sqrt{\dfrac{A_k}{aL_0}}\right\},\\[1em]
A_{k+1}=A_k+\alpha_{k+1}.
\end{cases}
$$
Define
$$
K_0:= \left\lceil128\ell(1+L_1R_0)\log\left(\frac{[64\ell(1+L_1R_0)]^2}{aL_0A_0}\right)\right\rceil.
$$
Then, for every $K\ge K_0$, $ A_K \ge\frac{4(K-K_0)^2}{25aL_0}.$ Moreover, for every $i\ge1$,
$$
\alpha_i \le \frac{i+1}{2aL_0},
$$
and, for every $K\ge1$,
$$
\sum_{i=1}^K\alpha_i^\alpha \le \frac{K(K+1)^\alpha}{2^\alpha(aL_0)^\alpha}.
$$
\end{lemma}

\subsection{Proofs}
\begin{proof}[Proof of \cref{lem:comparator2}]
Fix $u\in\mathbb R^d$ and define
$$
\Phi_i(u):=A_i\bigl(f(y_i)-f(u)\bigr)+\frac12\norm{z_i-u}^2.
$$
By convexity,
$$
\begin{cases}
f(y_{i-1})\ge f(x_i)+\langle\nabla f(x_i),y_{i-1}-x_i\rangle,\\[0.5em]
f(u)\ge f(x_i)+\langle\nabla f(x_i),u-x_i\rangle.
\end{cases}
$$
Multiplying the first inequality by $A_{i-1}$ and the second
by $\alpha_i$, and using $A_i=A_{i-1}+\alpha_i$ and $A_ix_i=A_{i-1}y_{i-1}+\alpha_i z_{i-1},$ we obtain
$$
A_if(x_i)-A_{i-1}f(y_{i-1})-\alpha_if(u) \le \alpha_i \langle\nabla f(x_i),z_{i-1}-u\rangle.
$$
We have $y_i-x_i=-\frac{\alpha_i^2}{A_i}g_i.$ Then, by the assumed descent inequality,
$$
A_i\bigl(f(y_i)-f(x_i)\bigr) \le-\alpha_i^2 \langle\nabla f(x_i),g_i\rangle
+\frac{\alpha_i^2}{4}\norm{g_i}^2.
$$
Since $z_i=z_{i-1}-\alpha_i g_i$, we have
$$
\frac12\norm{z_i-u}^2-\frac12\norm{z_{i-1}-u}^2
=
-\alpha_i\langle g_i,z_{i-1}-u\rangle
+\frac{\alpha_i^2}{2}\norm{g_i}^2.
$$
We deduce that
\begin{align*}
\Phi_i(u)-\Phi_{i-1}(u) &= A_i\bigl(f(y_i)-f(u)\bigr)-A_{i-1}\bigl(f(y_{i-1})-f(u)\bigr)+\frac12\norm{z_i-u}^2-\frac12\norm{z_{i-1}-u}^2 \\
&= A_i f(y_i)-A_{i-1} f(y_{i-1})-\alpha_i f(u)+\frac12\norm{z_i-u}^2-\frac12\norm{z_{i-1}-u}^2 
\\
&\le
\alpha_i\langle\nabla f(x_i),z_{i-1}-u\rangle-\alpha_i^2\langle\nabla f(x_i),g_i\rangle
+\frac{\alpha_i^2}{4}\norm{g_i}^2  -\alpha_i\langle g_i,z_{i-1}-u\rangle
+\frac{\alpha_i^2}{2}\norm{g_i}^2 \\
&=\alpha_i\langle\nabla f(x_i)-g_i,z_{i-1}-u\rangle +\alpha_i^2\left(-\langle\nabla f(x_i),g_i\rangle
+\frac34\norm{g_i}^2\right) \\
&=\alpha_i\langle\theta_i,u-z_{i-1}\rangle+\alpha_i^2\left(\langle\theta_i,g_i\rangle-\frac14\norm{g_i}^2\right).
\end{align*}
Since
$
\langle\theta_i,g_i\rangle-\frac14\norm{g_i}^2=\norm{\theta_i}^2-\frac14\norm{g_i-2\theta_i}^2
\le\norm{\theta_i}^2,$ it follows that
$$
\Phi_i(u)-\Phi_{i-1}(u) \le \alpha_i\langle\theta_i,u-z_{i-1}\rangle
+\alpha_i^2\norm{\theta_i}^2.
$$

Summing over $i=1,\ldots,t,$ we obtain
$$
\Phi_t(u) \le \Phi_0(u)+\mathcal R_t(u),$$
where $\Phi_0(u) =A_0\bigl(f(v_0)-f(u)\bigr)+\frac12\norm{v_0-u}^2.$
\end{proof}

\begin{proof}[Proof of \cref{lem:one-phase-descent}]
Let $j=1,\ldots,t$. By \cref{lem:comparator2}
with $u=x^\star,$ we have
$$
A_j\bigl(f(y_j)-f^\star\bigr)+\frac12\norm{z_j-x^\star}^2\le A_0\Delta_0
+\frac12\norm{x_0-x^\star}^2+\mathcal R_j(x^\star)\le\frac{65}{64}R_0^2.
$$
For $j=0$, we have $A_0\Delta_0\le R_0^2/64$ and $\norm{z_0-x^\star}=\norm{x_0-x^\star}\le R_0$.
Therefore, for all $j=0,\ldots,t$, we have
$$
\begin{cases}
f(y_j)-f^\star \le \dfrac{2R_0^2}{A_j},\\[0.8em]
\norm{z_j-x^\star} \le\dfrac32R_0.
\end{cases}
$$
Moreover, $A_ty_t=A_0x_0+\sum_{j=1}^t\alpha_j z_j$ and $A_t=A_0+\sum_{j=1}^t\alpha_j.$
Then
$$
y_t-x^\star =\frac{A_0}{A_t}(x_0-x^\star)+\sum_{j=1}^t\frac{\alpha_j}{A_t}(z_j-x^\star).
$$
This implies that $\norm{y_t-x^\star}\le\frac32R_0,$ and hence 
$$
\norm{z_t-y_t}\le 3R_0.
$$
It remains to show that the local smoothness around $x_{t+1}$ is small
enough compared with $A_{t+1}/\alpha_{t+1}^2$.

By the assumptions on the weights,
$$
\frac{\alpha_{t+1}}{A_t} \le \frac{1}{64\ell(1+L_1R_0)} \qquad\text{and}\qquad
\frac{A_t}{\alpha_{t+1}^2}\ge aL_0.
$$
Moreover,
$$
\frac{A_t}{\alpha_{t+1}^2}=\frac{1}{A_t}\left(\frac{A_t}{\alpha_{t+1}}\right)^2\ge \frac{64^2\ell^2(1+L_1R_0)^2}{A_t}.
$$
We first control the gradient at $y_t$. By \cref{lem:self-bounding-norm-form}
and $\nu\ge1/2$,
$$
\norm{\nabla f(y_t)} \le 4L_1\bigl(f(y_t)-f^\star\bigr) + 2\sqrt{L_0\bigl(f(y_t)-f^\star\bigr)}.
$$

Using $f(y_t)-f^\star\le2R_0^2/A_t$, we obtain $
\norm{\nabla f(y_t)} \le \frac{8L_1R_0^2}{A_t}+2\sqrt{2}R_0\sqrt{\frac{L_0}{A_t}}.$
Therefore,
\begin{align*}
L_0+L_1\norm{\nabla f(y_t)} &\le L_0 +\frac{8L_1^2R_0^2}{A_t} +2\sqrt{2}L_1R_0\sqrt{\frac{L_0}{A_t}}\\
&\le\left(\frac1a+\frac{1}{512\ell^2} +\frac{\sqrt{2}}{32\ell\sqrt a}\right)\frac{A_t}{\alpha_{t+1}^2}\\
&\le \frac{21}{512}\frac{A_t}{\alpha_{t+1}^2},
\end{align*}
where the last inequality uses $a\ge32$ and $\ell\ge1$.

We next transfer this bound from $y_t$ to $x_{t+1}$. We have $
x_{t+1}-y_t=\frac{\alpha_{t+1}}{A_{t+1}}(z_t-y_t).$
Therefore, using $\norm{z_t-y_t}\le3R_0$, we obtain
$$
L_1\norm{x_{t+1}-y_t} \le 3L_1R_0\frac{\alpha_{t+1}}{A_t}\le \frac{3}{64}.
$$
Thus, by \cref{lem:gradient-growth},
\begin{align*}
L_0+L_1\norm{\nabla f(x_{t+1})} &\le e^{L_1\norm{x_{t+1}-y_t}}
\bigl(L_0+L_1\norm{\nabla f(y_t)}\bigr)\\
&\le e^{3/64}\frac{21}{512} \frac{A_t}{\alpha_{t+1}^2}\\
&\le \frac18 \frac{A_t}{\alpha_{t+1}^2}\\
&\le \frac18 \frac{A_{t+1}}{\alpha_{t+1}^2}.
\end{align*}
It holds that $2\bigl(L_0+L_1\norm{\nabla f(x_{t+1})}\bigr) \le
\frac{A_{t+1}}{4\alpha_{t+1}^2}.$

We have $
y_{t+1}-x_{t+1} =-\frac{\alpha_{t+1}^2}{A_{t+1}}g_{t+1}$
and $\norm{g_{t+1}}\le\lambda_{t+1}$, then
\begin{align*}
L_1\norm{y_{t+1}-x_{t+1}}&\le\frac{L_1\alpha_{t+1}^2}{A_{t+1}}\lambda_{t+1}\\
&= \frac{L_1R_0}{\ell}\frac{\alpha_{t+1}}{A_{t+1}}\\
&\le\frac{L_1R_0}{64\ell^2(1+L_1R_0)} \le 1.
\end{align*}
By  \cref{lem:local-smoothness-short-segment}, we obtain 
\begin{align*}
f(y_{t+1})
&\le
f(x_{t+1}) +\langle\nabla f(x_{t+1}),y_{t+1}-x_{t+1}\rangle +2\bigl(L_0+L_1\norm{\nabla f(x_{t+1})}\bigr)\norm{y_{t+1}-x_{t+1}}^2\\
&\le f(x_{t+1}) +\langle\nabla f(x_{t+1}),y_{t+1}-x_{t+1}\rangle +\frac{A_{t+1}}{4\alpha_{t+1}^2} \norm{y_{t+1}-x_{t+1}}^2.
\end{align*}
\end{proof}

\begin{proof}[Proof of \cref{lem:one-phase-gradient}]
From the proof of \cref{lem:one-phase-descent}, we have $\norm{z_t-y_t}\le3R_0$ and
$$
\norm{\nabla f(y_t)} \le\frac{8L_1R_0^2}{A_t}+ 2\sqrt{2}R_0\sqrt{\frac{L_0}{A_t}}.
$$
Therefore, using the assumptions on the weights,
\begin{align*}
\frac{\alpha_{t+1}\ell}{R_0}\norm{\nabla f(y_t)}
&\le 8\ell L_1R_0\frac{\alpha_{t+1}}{A_t}+ 2\sqrt{2}\ell\alpha_{t+1}\sqrt{\frac{L_0}{A_t}}\\
&\le
\frac18+\frac{2\sqrt{2}\ell}{\sqrt a}\\
&\le \frac38,
\end{align*}
where the last inequality follows from $a\ge128\ell^2$.

We next transfer this bound from $y_t$ to $x_{t+1}$. Since $x_{t+1}-y_t=\frac{\alpha_{t+1}}{A_{t+1}}(z_t-y_t),$
and $\norm{z_t-y_t}\le3R_0$, we have
$$
L_1\norm{x_{t+1}-y_t}\le 3L_1R_0\frac{\alpha_{t+1}}{A_t} \le \frac{3}{64}.
$$
Hence, by \cref{lem:gradient-growth},
$$
L_0+L_1\norm{\nabla f(x_{t+1})} \le e^{L_1\norm{x_{t+1}-y_t}} \bigl(L_0+L_1\norm{\nabla f(y_t)}\bigr),
$$
and therefore
$$
\norm{\nabla f(x_{t+1})} \le e^{L_1\norm{x_{t+1}-y_t}} \norm{\nabla f(y_t)} + \frac{e^{L_1\norm{x_{t+1}-y_t}}-1}{L_1}L_0.
$$
Using $L_1\norm{x_{t+1}-y_t}\le3/64$ and
$e^x-1\le xe^x$ for $x\ge0$, we obtain
$$
\frac{e^{L_1\norm{x_{t+1}-y_t}}-1}{L_1}L_0\le e^{3/64}L_0\norm{x_{t+1}-y_t} \le
3e^{3/64}L_0R_0\frac{\alpha_{t+1}}{A_t}.
$$
Thus,
\begin{align*}
\frac{\alpha_{t+1}\ell}{R_0}\norm{\nabla f(x_{t+1})} &\le e^{3/64}\frac{\alpha_{t+1}\ell}{R_0}
\norm{\nabla f(y_t)}+ 3e^{3/64}\ell\frac{L_0\alpha_{t+1}^2}{A_t}\\
&\le e^{3/64}\left( \frac38+\frac{3\ell}{a}\right)\\
&< \frac12,
\end{align*}
where the last inequality follows from $a\ge128\ell^2$, $\ell\ge1$, and $e^{3/64}<16/15$.
Consequently,
$$
\norm{\nabla f(x_{t+1})} \le \frac{R_0}{2\alpha_{t+1}\ell} =\frac{\lambda_{t+1}}{2}.
$$
\end{proof}

\begin{proof}[Proof of \cref{lem:one-phase-residual}]
For $t=1,\ldots,K$, define
$$
S_t:= \frac{\sigma^\alpha\ell^{\alpha-1}}{R_0^\alpha} \sum_{i=1}^t\alpha_i^\alpha.
$$
By assumption, we have $S_t \le\frac1{512}.$

Using the proof of \cref{lem:one-phase-descent}, for every
$i=1,\ldots,t$, on $\mathcal E_{i-1}$,
$$
\norm{x^\star-z_{i-1}}\le\frac32R_0.
$$

We first control the predictable term $B_t$. Using the clipping bias
and conditional second-moment bounds,
\begin{align*}
B_t &\le \sum_{i=1}^t \mathbf 1_{\mathcal E_{i-1}} \left(\frac32\,2^\alpha\sigma^\alpha R_0
\alpha_i\lambda_i^{1-\alpha}+2\sigma^\alpha\alpha_i^2\lambda_i^{2-\alpha}\right).
\end{align*}
Since $\lambda_i=R_0/(\alpha_i\ell)$,
$$
\begin{cases}
\alpha_i\lambda_i^{1-\alpha} =\alpha_i^\alpha R_0^{1-\alpha}\ell^{\alpha-1},\\[0.6em]
\alpha_i^2\lambda_i^{2-\alpha}= \alpha_i^\alpha R_0^{2-\alpha}\ell^{\alpha-2}.
\end{cases}
$$
Therefore,
$$
B_t \le \left(3\cdot2^{\alpha-1}+\frac{2}{\ell} \right) R_0^2S_t\le 8R_0^2S_t \le\frac{R_0^2}{64},
$$
where we used $2^\alpha\le4$ and $\ell\ge1$.

We next control the martingale term $M_t$. Let $X_i$ denote its
$i$-th increment. On $\mathcal E_{i-1}$,
$$
\begin{cases}
\norm{\theta_i-\mathbb E[\theta_i\mid\mathcal F_{i-1}]} \le 3\lambda_i,\\[0.8em]
\left|\norm{\theta_i}^2 -\mathbb E[\norm{\theta_i}^2\mid\mathcal F_{i-1}]\right|
\le \dfrac94\lambda_i^2.
\end{cases}
$$
Therefore,
\begin{align*}
|X_i| &\le 3\alpha_i\lambda_i\norm{x^\star-z_{i-1}} +\frac94\alpha_i^2\lambda_i^2\\
&\le \frac92\alpha_i\lambda_iR_0 + \frac94\alpha_i^2\lambda_i^2\\
&= \left( \frac{9}{2\ell} + \frac{9}{4\ell^2}\right)R_0^2 \le \frac{19}{4\ell}R_0^2,
\end{align*}
where the last inequality uses $\ell\ge8\log(4K/\delta)>9$.

We now bound the conditional variances. On $\mathcal E_{i-1}$,
$$
\begin{cases}
\mathbb E\left[ \norm{\theta_i-\mathbb E[\theta_i\mid\mathcal F_{i-1}]}^2
\mid\mathcal F_{i-1} \right] \le \mathbb E[\norm{\theta_i}^2\mid\mathcal F_{i-1}]
\le 2\sigma^\alpha\lambda_i^{2-\alpha},\\[1em]
\mathbb E[\norm{\theta_i}^4\mid\mathcal F_{i-1}] \le \dfrac94\lambda_i^2
\mathbb E[\norm{\theta_i}^2\mid\mathcal F_{i-1}] \le \dfrac92\sigma^\alpha\lambda_i^{4-\alpha}.
\end{cases}
$$
Since $\mathbf 1_{\mathcal E_{i-1}}$ is
$\mathcal F_{i-1}$-measurable, using $(a+b)^2\le2a^2+2b^2$ and
$\norm{x^\star-z_{i-1}}\le\frac32R_0$, we obtain
$$
\mathbb E[X_i^2\mid\mathcal F_{i-1}]
\le
9\sigma^\alpha\alpha_i^2R_0^2\lambda_i^{2-\alpha}
+
9\sigma^\alpha\alpha_i^4\lambda_i^{4-\alpha}.
$$
Since $\lambda_i=R_0/(\alpha_i\ell)$, summing over $i=1,\ldots,t,$ we obtain
$$
\sum_{i=1}^t \mathbb E[X_i^2\mid\mathcal F_{i-1}]\le 9R_0^4 \left( \frac1\ell+\frac1{\ell^3}
\right)S_t \le \frac{10R_0^4}{512\ell},$$ 
where we again used $\ell>9$ and $S_t\le 1/512$.

Fix $t\in\{1,\ldots,K\}$. Applying \cref{cor:freedman-hp} with failure
probability $\delta/K$, we obtain, with probability at least
$1-\delta/K$,
\begin{align*}
|M_t| &\le R_0^2 \left( \sqrt{ \frac{20}{512\ell} \log\frac{2K}{\delta} }+\frac{19}{6\ell}
\log\frac{2K}{\delta}\right)\\
&\le R_0^2 \left( \frac{\sqrt5}{32} + \frac{19}{48} \right)\\
&< \frac{31}{64}R_0^2,
\end{align*}
where we used $ \ell\ge8\log\frac{4K}{\delta} \ge 8\log\frac{2K}{\delta}.$

Taking a union bound over $t=1,\ldots,K$, with probability at least
$1-\delta$, simultaneously for every $t=1,\ldots,K$,
$$
|M_t|<\frac{31}{64}R_0^2.
$$
On this event,
$$
\widehat{\mathcal R}_t(x^\star)=M_t+B_t\le|M_t|+B_t<\frac{31}{64}R_0^2+\frac1{64}R_0^2=\frac{R_0^2}{2}.
$$

It remains to show that $\mathcal E_K$ holds on this event.
We argue by induction. For $t=0$, there are no conditions, so
$\mathbf 1_{\mathcal E_0}=1$. Suppose that $\mathcal E_{t-1}$ holds
for some $t\in\{1,\ldots,K\}$. Applying
\cref{lem:one-phase-descent} with $t-1$ gives
$$
f(y_t) \le f(x_t)+\langle\nabla f(x_t),y_t-x_t\rangle +\frac{A_t}{4\alpha_t^2}\norm{y_t-x_t}^2.
$$
Moreover, since $\mathcal E_{t-1}$ holds,
$$
\widehat{\mathcal R}_t(x^\star) = \mathcal R_t(x^\star) \le \frac{R_0^2}{2}.
$$
Thus all the conditions defining $\mathcal E_t$ hold. By induction,
$\mathcal E_K$ holds.
\end{proof}

\begin{proof}[Proof of \cref{lem:one-phase-weights}]
We first show that, for every $k\ge0$,
$$
\sqrt{aL_0A_{k+1}}-\sqrt{aL_0A_k}\le\frac12.
$$
If $A_k\le\frac{[64\ell(1+L_1R_0)]^2}{aL_0},$ then  $A_{k+1}=\left(1+\frac{1}{64\ell(1+L_1R_0)}\right)A_k.
$ Therefore,
$$
\sqrt{aL_0A_{k+1}}-\sqrt{aL_0A_k} = \frac{\sqrt{aL_0A_k}/[64\ell(1+L_1R_0)]}{\sqrt{1+1/[64\ell(1+L_1R_0)]}+1} \le \frac12.
$$
If $A_k\ge\frac{[64\ell(1+L_1R_0)]^2}{aL_0},$ then $A_{k+1}=A_k+\sqrt{\frac{A_k}{aL_0}}.$
Therefore
$$
\begin{aligned}
\sqrt{aL_0A_{k+1}}-\sqrt{aL_0A_k}&=\frac{aL_0(A_{k+1}-A_k)}{\sqrt{aL_0A_{k+1}}+\sqrt{aL_0A_k}}\\
&=\frac{\sqrt{aL_0A_k}}{\sqrt{aL_0A_{k+1}}+\sqrt{aL_0A_k}} \le\frac12.
\end{aligned}
$$
Since $A_0\le1/(aL_0)$, it follows that $\sqrt{aL_0A_k}\le1+\frac{k}{2}=\frac{k+2}{2}.$
Using $\alpha_{k+1} \le \sqrt{\frac{A_k}{aL_0}}=\frac{\sqrt{aL_0A_k}}{aL_0},$
we obtain
$$
\alpha_{k+1} \le \frac{k+2}{2aL_0}.
$$
Thus, for every $i\ge1$, we have $\alpha_i\le\frac{i+1}{2aL_0}.$

Consequently,
$$
\sum_{i=1}^K\alpha_i^\alpha \le \frac{1}{2^\alpha(aL_0)^\alpha} \sum_{i=1}^K(i+1)^\alpha
\le \frac{K(K+1)^\alpha}{2^\alpha(aL_0)^\alpha}.
$$

It remains to lower-bound $A_K$. Define
$$
m:= \min\left\{ k\ge0:A_k\ge\frac{[64\ell(1+L_1R_0)]^2}{aL_0}\right\}.
$$

Since $A_0>0$, if $A_k<\frac{[64\ell(1+L_1R_0)]^2}{aL_0}$ for every $k\ge0$, then
$$
A_k= \left( 1+\frac{1}{64\ell(1+L_1R_0)}\right)^kA_0\to\infty,
$$
which is impossible. Hence $m$ is well-defined.

We first show that $m\le K_0$. Suppose, by contradiction, that
$m>K_0$. Then, for every $k=0,\ldots,K_0-1$, we have $A_k<\frac{[64\ell(1+L_1R_0)]^2}{aL_0},$ and therefore
$$
A_{K_0} =\left(1+\frac{1}{64\ell(1+L_1R_0)}\right)^{K_0}A_0.
$$
Since $\log(1+x)\ge x/2$ for all $x\in[0,1]$, it holds that
$$
\log\left( 1+\frac{1}{64\ell(1+L_1R_0)}\right) \ge \frac{1}{128\ell(1+L_1R_0)}.
$$
Thus,
$$
\begin{aligned}
A_{K_0} &= A_0\exp\left( K_0 \log\left( 1+\frac{1}{64\ell(1+L_1R_0)}\right)\right)\\
&\ge A_0\exp\left( \frac{K_0}{128\ell(1+L_1R_0)}\right) \\
&\ge A_0\exp\left(\log\left(\frac{[64\ell(1+L_1R_0)]^2}{aL_0A_0}\right)\right)\\
&= \frac{[64\ell(1+L_1R_0)]^2}{aL_0},
\end{aligned}
$$
which contradicts $m>K_0$. Thus $m\le K_0$. For every $k\ge m$, we have
$$
A_{k+1}= A_k+\sqrt{\frac{A_k}{aL_0}}.
$$
Moreover, $\sqrt{aL_0A_k} \ge 64\ell(1+L_1R_0) \ge 1.$ Therefore,
$$
\sqrt{aL_0A_{k+1}}-\sqrt{aL_0A_k}=\frac{1}{\sqrt{1+1/\sqrt{aL_0A_k}}+1} \ge \frac{1}{\sqrt2+1}>\frac25.
$$
Thus, for every $K\ge K_0$,
$$
\sqrt{aL_0A_K} \ge \frac25(K-m) \ge \frac25(K-K_0).
$$
We deduce that
$$
A_K \ge \frac{4(K-K_0)^2}{25aL_0}.
$$
\end{proof}

\begin{proof}[Proof of \cref{thm:one-phase}]
By the definition of the weights, for every $k=0,\ldots,K-1$,
$$
\begin{cases}
\displaystyle \frac{\alpha_{k+1}}{A_k}\le\frac{1}{64\ell(1+L_1R_0)},\\[1em]
\displaystyle A_k\ge aL_0\alpha_{k+1}^2.
\end{cases}
$$
Moreover, $A_0\Delta_0\le A_0\bar\Delta_0\le R_0^2/64$, $\ell=8\log(4K/\delta)$, and $a\ge128\ell^2$. It remains to verify the last condition of \cref{lem:one-phase-residual}. By  \cref{lem:one-phase-weights} and $K+1\le2K$,
$$
\sum_{i=1}^K\alpha_i^\alpha
\le\frac{K(K+1)^\alpha}{2^\alpha(aL_0)^\alpha}
\le\frac{K^{\alpha+1}}{(aL_0)^\alpha}.
$$
By definition of $a$, it holds that
$$
\frac{\sigma^\alpha\ell^{\alpha-1}}{R_0^\alpha}\sum_{i=1}^K\alpha_i^\alpha
\le\frac{\sigma^\alpha\ell^{\alpha-1}K^{\alpha+1}}{R_0^\alpha(aL_0)^\alpha}
\le\frac1{512}.
$$
Therefore, all the conditions of \cref{lem:one-phase-residual} are satisfied. Then, with probability at least $1-\delta$, the event $\mathcal E_K$ holds and
$$
\mathcal R_K(x^\star)=\widehat{\mathcal R}_K(x^\star)\le\frac{R_0^2}{2}.
$$
On this event, by applying \cref{lem:comparator2} with $u=x^\star$, we obtain
$$
A_K\bigl(f(y_K)-f^\star\bigr)+\frac12\norm{z_K-x^\star}^2
\le A_0\Delta_0+\frac12\norm{x_0-x^\star}^2+\mathcal R_K(x^\star)
\le\frac{65}{64}R_0^2.
$$
Consequently,
$$
f(y_K)-f^\star\le\frac{65R_0^2}{64A_K}.
$$
Since $K\ge2K_0$,  \cref{lem:one-phase-weights} yields
$$
A_K\ge\frac{4(K-K_0)^2}{25aL_0}\ge\frac{K^2}{25aL_0},
$$
and therefore
$$
f(y_K)-f^\star\le\frac{1625}{64}\frac{aL_0R_0^2}{K^2}.
$$
Using the definition of $a$ and $\max\{x,y\}\le x+y$, we obtain
$$
f(y_K)-f^\star
\le3250\frac{L_0R_0^2\ell^2}{K^2}
+\frac{1625}{64}\,512^{1/\alpha}
\frac{\sigma R_0\ell^{(\alpha-1)/\alpha}}{K^{(\alpha-1)/\alpha}}.
$$

\end{proof}

\end{document}